\documentclass[11pt,reqno]{amsart}
\usepackage{amsmath,amssymb,mathrsfs,amsthm,amsfonts}
\usepackage[inline]{enumitem} 
\usepackage[usenames,dvipsnames]{xcolor}
\usepackage{hyperref}
\usepackage{comment}
\usepackage{bbm}
\usepackage{graphicx}
\hypersetup{%
	colorlinks=true, linkcolor=blue,
	citecolor=ForestGreen
}
\usepackage[paper=letterpaper,margin=1in]{geometry}

\newtheorem{theorem}{Theorem}[section]
\newtheorem{lemma}[theorem]{Lemma}
\newtheorem{proposition}[theorem]{Proposition}

\theoremstyle{definition}

\newtheorem{remark}[theorem]{Remark}
\numberwithin{equation}{section}
\allowdisplaybreaks
\usepackage{acronym}

\acrodef{KPZ}{Kardar--Parisi--Zhang}
\acrodef{SHE}{Stochastic Heat Equation}
\acrodef{LDP}{Large Deviation Principle}

\renewcommand{\P}{\mathbb{P}}	
\newcommand{\Ex}{\mathbb{E}}	
\renewcommand{\d}{\mathrm{d}}	
\newcommand{\ind}{\mathbf{1}}	
\newcommand{\tr}{\mathrm{tr}} 	
\newcommand{\rl}{\operatorname{Re}}

\newcommand{\g}{\rho}
\newcommand{\C}{\mathfrak{C}}

\newcommand{\norm}[1]{\Vert#1\Vert}

\newcommand{\Con}{\mathrm{C}} 

\newcommand{\R}{\mathbb{R}} 
\newcommand{\Z}{\mathbb{Z}} 
\renewcommand{\i}{\mathbf{i}}

\newcommand{\e}{\varepsilon}

\newcommand{\im}{\operatorname{Im}}
\newcommand{\hq}{h_q}
\newcommand{\bq}{B_q}
\newcommand{\fq}{F_q}
\newcommand{\calA}{\mathcal{A}}
\newcommand{\calB}{\mathcal{B}}
\newcommand{\calH}{\mathcal{H}}

\newcommand{\calR}{\mathcal{R}}

\newcommand{\til}{\widetilde}

\title[Upper-Tail LDP for ASEP]{Upper-tail Large Deviation Principle for the ASEP}

\author[S.\ Das]{Sayan Das}
\address{S.\ Das,
	Department of Mathematics, Columbia University,
	\newline\hphantom{\quad \ \ S. Das}
	2990 Broadway, New York, NY 10027 USA
}
\email{sayan.das@columbia.edu}

\author[W.\ Zhu]{Weitao Zhu}
\address{W.\ Zhu,
	Department of Mathematics, Columbia University,
	\newline\hphantom{\quad \ \ S. Das}
	2990 Broadway, New York, NY 10027 USA
}
\email{weitao.zhu@columbia.edu}
\subjclass[2010]{%
	Primary 60F10,		
	Secondary 82C22.  	
}
\keywords{%
	ASEP, Lyapunov exponents, large deviations, Fredholm determinants.
}

\begin{document}
	\begin{abstract} We consider the asymmetric simple exclusion process (ASEP) on $\Z$ started from step initial data and obtain the exact Lyapunov exponents for $H_0(t)$, the integrated current of ASEP. As a corollary, we derive an explicit formula for the upper-tail large deviation rate function for $-H_0(t)$. Our result matches with the rate function for the integrated current of the totally asymmetric simple exclusion process (TASEP) obtained in \cite{joh}.	
	\end{abstract}
	\maketitle
	
	\section{Introduction}
\subsection{The ASEP and main results} In this paper, we study the upper-tail Large Deviation Principle (LDP) of the \textit{asymmetric simple exclusion process} (ASEP) with step initial data. The ASEP is a continuous-time Markov chain on particle configurations $\textbf{x} = (\textbf{x}_1 > \textbf{x}_2 > \cdots)$ in $\Z$.  The process can be described as follows. Each site $i\in \Z$ can be occupied by at most one particle, which has an independent exponential clock with exponential waiting time of mean $1$. When the clock rings, the particle jumps to the right with probability $q$ or to the left with probability $p=1-q$. However, the jump is only permissible when the target site is unoccupied. For our purposes, it suffices to consider configurations with a rightmost particle. At any time $t \in \R_{>0}$, the process has the configuration  $x(t)=(x_1(t)>x_2(t)>\cdots)$ in $\Z$, where $x_j(t)$ denotes the location of the $j$-th rightmost particle at this time. {Appearing first in the biology work of Macdonald, Gibbs, and Pipkin \cite{pip} and introduced  to the mathematics community two years later by \cite{spitzer},} the ASEP has since become the ``default stochastic model to study transport phenomena", including mass transport, traffic flow, queueing behavior, driven lattices and turbulence. We refer to \cite{bcs,liggett,liggett2,spohn} for the mathematical study of and related to the ASEP. 

When $q=1,$ we obtain the \textit{totally asymmetric simple exclusion process} (TASEP), which allows jumps only to the right. It connects to several other physical systems such as the exponential last-passage percolation, zero-temperature directed polymer in a random environment, the corner growth process and is known to possess complete determinantal structure (\textit{free-fermionicity}). We refer the readers to \cite{joh,liggett,liggett2,pra} and the references therein for more thorough treatises of the TASEP.

The dynamics of ASEP are uniquely determined once we specify its initial state. In the present paper, we restrict our attention to the ASEP  started from the \textit{step} initial configuration, i.e. $x_j(0)=-j$, $j=1,2,\ldots$. We set $\gamma = q-p$ and assume $q>\frac12$, i.e., ASEP has a drift to the right. An observable of interest in ASEP is $H_0(t)$, the integrated current through 0 which is defined as:
\begin{align}\label{def:ht}
	H_0(t) := \mbox{ the number of particles to the right of zero at time }t.
\end{align}
$H_0(t)$ can also be interpreted as the one-dimensional height function of the interface growth of the ASEP and thus carries significance in the broader context of the Kardar-Parisi-Zhang (KPZ) universality class. We will elaborate on the connection to KPZ universality class later in Section \ref{sec:pre}. As a well-known random growth model itself, the large-time behaviors of ASEP with step initial conditions have been well-studied.  Indeed, it is known \cite[Chapter VIII, Theorem 5.12]{liggett} that the current satisfies the following strong law of large numbers:
\begin{align*}
\tfrac1t{H_0\big(\tfrac{t}\gamma\big)} \rightarrow \tfrac{1}{4}, \mbox{ almost surely as } t\to \infty.
\end{align*}

The strong law has been later complemented by fluctuation results in the seminal works by Tracy and Widom. In a series of papers \cite{tw3}, \cite{tw1} \cite{tw2}, Tracy and Widom exploit the integrability of ASEP with step initial data and establish via contour analysis that $H_0(t)$ when centered by $\frac{t}{4}$ has typical deviations of the order $t^{1/3}$ and has the following asymptotic fluctuations:
\begin{align}\label{eq:clt2}
{\tfrac{1}{t^{1/3}}2^{4/3}\big(-H_0\big(\tfrac{t}\gamma\big) + \tfrac{t}{4}\big) \implies   \xi_{\operatorname{GUE}},}
\end{align}
where $\xi_{\operatorname{GUE}}$ is the GUE Tracy-Widom distribution
  \cite{tw4}. When $q = 1$, \eqref{eq:clt2} recovers the same result on  TASEP, which has been proved earlier by \cite{joh}.

Given the existing fluctuation results on the ASEP with step initial data, it is natural to inquire into its {Large Deviation Principle (LDP)}. Namely, we seek to find the probability of when the event $-H_0(\frac{t}{\gamma})+\frac{t}{4}$ has deviations of order $t$. Intriguingly, one expects the lower- and upper-tail LDPs to have different speeds: the upper-tail deviation is expected to occur at speed $t$ whereas the lower-tail has speed $t^2$:
\begin{align*}\tag{Lower Tail}
{\mathbb{P}\left(-H_0\big(\tfrac{t}\gamma\big)+\tfrac{t}{4} < -\tfrac{t}{4}y\right)	\approx e^{-t^2\Phi_{-}(y)};} 
\end{align*}
\begin{align*}\tag{Upper Tail}
{\mathbb{P}\left(-H_0\big(\tfrac{t}\gamma\big)+\tfrac{t}{4} > +\tfrac{t}{4}y\right)\approx e^{-t\Phi_{+}(y)}.}
\end{align*}
Thus, the upper tail corresponds to ASEP being ``too slow" while the lower tail corresponds to ASEP being ``too fast". Heuristically, we can make sense of such speed differentials. Because of the nature of the exclusion process, when a \textit{single} particle is moving slower than the usual, it forces \textit{all} the particles on the left of it to be automatically slower. Hence ASEP becomes slow if \textit{only one} particle is moving slow. This event has probability of the order $\exp(-O(t))$. However, in order to ensure that there are many particles on the right side of origin (this corresponds to ASEP being fast), it requires a large number of particles to move fast \textit{simulatenously}. This event is much more unlikely and happens with probability $\exp(-O(t^2))$.

In this article, we focus on the \textit{upper-tail} deviations of the ASEP with step initial data and present the first proof of the ASEP upper-tail LDP on the \textit{complete} real line. Consider ASEP with $q\in (\frac12,1)$ and set $p=1-q$ and $\tau=p/q\in (0,1)$. Our first theorem computes the $s$th-\textit{Lyapunov exponent} of $\tau^{H_0(t)}$, which is the limit of  the logarithm of $\Ex[\tau^{sH_0(t)}]$ scaled by time:
	\begin{theorem} \label{thm:frac_mom}
		For $s\in (0,\infty)$ we have
		\begin{align}\label{eq:exp} 
			\lim_{t\to \infty} \frac1t\log \Ex [\tau^{sH_0(t)}]=-\hq(s)=:-(q-p)\frac{1-\tau^{\frac{s}2}}{1+\tau^{\frac{s}2}}.
		\end{align}
	\end{theorem}

It is well known (see Proposition 1.12 in \cite{gl20} for example) that the \textit{upper-tail} large deviation principle of the stochastic process $\log \tau^{H_0(t)}$ is the Legendre-Fenchel dual of the Lyapunov exponent in \eqref{eq:exp}. Since $\tau<1$, as a corollary, we obtain the following \textit{upper-tail} large deviation rate function for $-H_0(t)$.
\begin{theorem}\label{thm:ldp}
	{ For any $y\in (0,1)$ we have
		\begin{align}\label{eq:ldp}
			\lim_{t\to\infty}\frac1t\log\P\left(-H_0\big(\tfrac{t}\gamma\big)+\tfrac{t}{4} > \tfrac{t}{4}y\right)=-[\sqrt{y}-(1-y)\tanh^{-1}(\sqrt{y})]=:-\Phi_{+}(y),
		\end{align}}
	where $\gamma=2q-1$. Furthermore, we have the following asymptotics near zero:
	\begin{align}\label{eq:asy}
		\lim_{y\to 0^+} y^{-3/2}\Phi_{+}(y)=\tfrac23.
	\end{align}
	\end{theorem}
\begin{figure}[h!]
	\begin{center}
		\includegraphics[width=7cm]{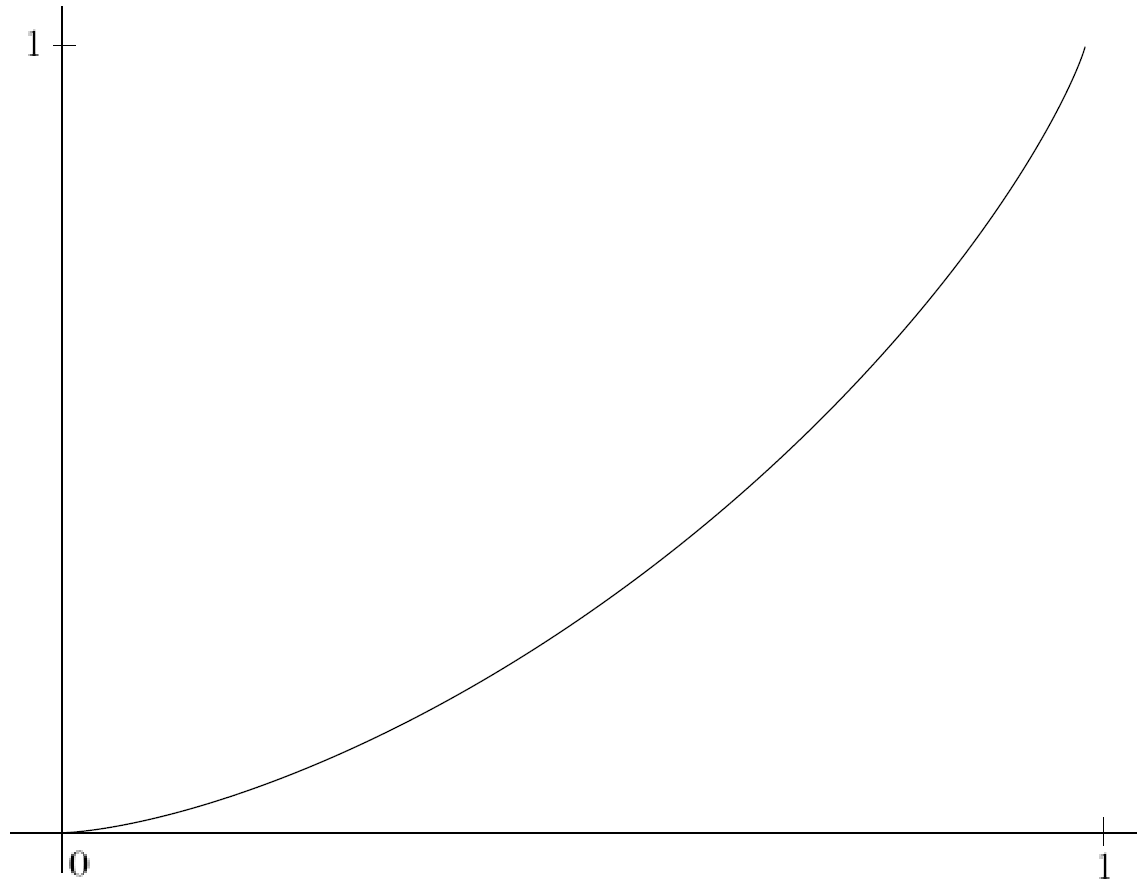} \ \ 
		\includegraphics[width=7cm]{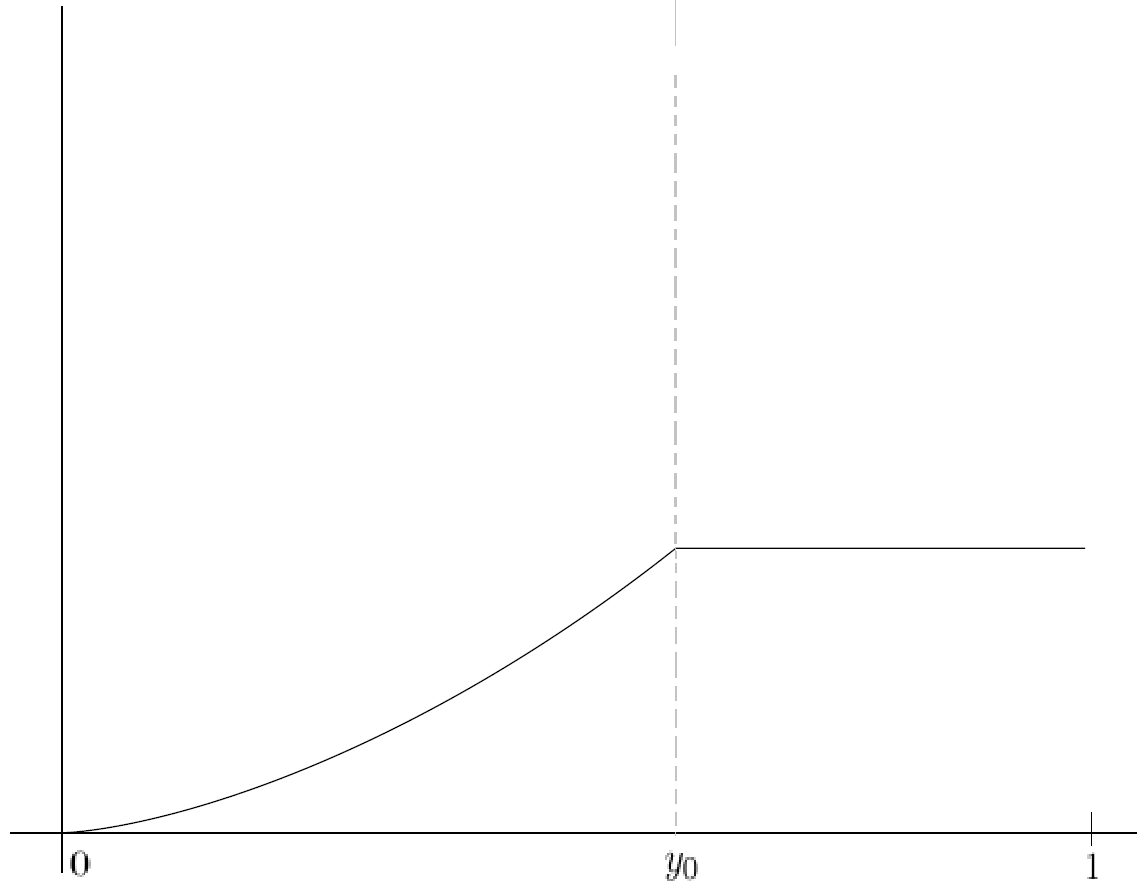}
		\vspace{-2mm}
		\caption{The figure on the left is the plot of $\Phi_{+}(y)$. The right one is the plot of $\til\Phi_{+}(y)$.}
		\label{tphi}
	\end{center}
\end{figure}

\begin{remark} Note that our large deviation result is restricted to $y\in (0,1)$ as $\P(-H_0\big(\tfrac{t}\gamma\big)+\tfrac{t}{4} > \tfrac{t}{4}y)=0$ for $y\ge 1$. Furthermore, although Theorem \ref{thm:ldp} makes sense when $q=1$, one cannot recover it from Theorem \ref{thm:frac_mom}, which only makes sense for $\tau=(1-q)/q\in (0,1)$. However, as mentioned before, \cite{joh} has already settled the $q=1$ TASEP case and obtained the upper-tail rate function in a variational form. We will later show in Appendix \ref{app} that \cite{joh} variational formula for TASEP matches with our rate function in \eqref{eq:ldp}.
\end{remark}
\begin{remark} {Recently, the work \cite{dp} has obtained a one-sided large deviation bound for the upper tail of the ASEP. In particular, they showed
	\begin{align}
		\mathbb{P}\left(-H_0\big(\tfrac{t}\gamma\big)+\tfrac{t}{4} > \tfrac{t}{4}y\right) \le \Con e^{-t\til\Phi_{+}(y)},\quad y\in (0,1).
	\end{align}
The function $\til\Phi_{+}$ coincides with the correct rate function $\Phi_{+}$ defined in \eqref{eq:ldp} only for $y \le y_0:= \frac{1-2\sqrt{q(1-q)}}{1+2\sqrt{q(1-q)}}$, as captured by Figure \ref{tphi}. We  will further compare and contrast our results and method with \cite{dp} later in Section \ref{sec:pre}.}
\end{remark}
\begin{remark} For $y$ small enough, following \eqref{eq:clt2} and upper tail decay of GUE Tracy-Widom distribution \cite{dumaz}, one expects {$$\P\left(-H_0\big(\tfrac{t}\gamma\big)+\tfrac{t}{4} > \tfrac{t}{4}y\right) \approx \P(\xi_{\operatorname{GUE}} >2^{-2/3}yt^{2/3}) \approx e^{-\frac{2}3y^{3/2}t}$$}
	Thus the asymptotics in \eqref{eq:asy} shows that $\Phi_{+}(y)$ indeed recovers the expected GUE Tracy-Widom tails as $y\to 0^+$.
\end{remark}

	\subsection{Sketch of proof}\label{sec:ske} In this section we present a sketch of the proof of our main results. {As explained before, Theorem \ref{thm:ldp} can be obtained from Theorem \ref{thm:frac_mom} by standard Legendre-Fenchel transform technique. {So here we only give a brief account of the proof  idea of Theorem \ref{thm:frac_mom}. A more detailed overview of the proofs of our main results can be found in Section \ref{sec:fraccase}}.}
	
	The main component of our proof is the following $\tau$-Laplace transform formula for ${H_0(t)}$ that appears in	Theorem 5.3 in \cite{bcs}:
	
	\begin{theorem}[Theorem 5.3 in \cite{bcs}] \label{thm:laplace} Fix any $\delta\in (0,1)$. For $\zeta>0$ we have
		\begin{align}\label{eq:laplace}
			\Ex \left[F_q(\zeta\tau^{H_0(t)})\right]=\det(I+K_{\zeta,t}), \quad F_q(\zeta):=\prod_{n=0}^{\infty}\frac{1}{1+\zeta\tau^n}.
		\end{align}
	Here $\det(I + K_{\zeta, t})$ is the Fredholm determinant of  $K_{\zeta,t}: L^2(\C(\tau^{1-\frac{\delta}{2}})) \rightarrow L^2(\C(\tau^{1-\frac{\delta}{2}})),$ and $\C(\tau^{1-\frac{\delta}{2}})$ denotes a positively-oriented circular contour centered at 0 with radius $\tau^{1-\frac{\delta}{2}}.$
	The operator $K_{\zeta, t}$ is defined through the integral kernel
	\begin{align} \label{def: ker}
		K_{\zeta,t}(w,w') &: =\frac1{2\pi \i}\int\limits_{\delta-\i\infty}^{\delta+\i\infty} \Gamma(-u)\Gamma(1+u)\zeta^u	\frac{g_t({w})}{g_t({\tau^uw})}\frac{\d u}{w'-\tau^u w},\ \mbox{ for  }g_t(z)=\exp\left(\frac{(q-p)t}{1+\frac{z}{\tau}}\right).
	\end{align}
	\end{theorem}

\begin{remark}
	The original statement of the above theorem in \cite{bcs} appears in a much more general setup with general conditions on the contours. We will explain the choice of our contours stated above in Section \ref{sec:leading} and check that it satisfies the general criterion for contours as stated in Theorem 5.3 in \cite{bcs}.
\end{remark}

\noindent We next recall that the Fredholm determinant is defined as a series as follows.
\begin{align}
	\det(I+K_{\zeta,t}) & :=1+\sum_{L=1}^{\infty} \tr(K_{\zeta,t}^{\wedge L}) \label{eq:fdhm} \\ & := 1+\sum_{L=1}^{\infty} \frac{1}{L!}\int_{\C(\tau^{1-\frac{\delta}2})}\cdots \int_{\C(\tau^{1-\frac{\delta}2})} \det(K_{\zeta,t}(w_i,w_j))_{i,j=1}^{L}\prod_{i=1}^L \d w_i. \label{eq:f-series}
\end{align}
The notation $K_{\zeta,t}^{\wedge L}$ comes from the exterior algebra definition, which we refer to \cite{sim77} for more details. As a clarifying remark, we use this exterior algebra notation only for the simplicity of its expression and rely essentially on the definition in \eqref{eq:f-series} throughout the rest of the paper.

	To extract information on the fractional moments of $\tau^{H_0(t)}$, we combine the formula in \eqref{eq:laplace} with the following elementary identity, which is a generalized version of Lemma 1.4 in \cite{dt19}.

	\begin{lemma}\label{lm:frac_mom} Fix $n\in \Z_{>0}$ and $\alpha\in [0,1)$. Let $U$ be a nonnegative random variable with finite $n$-th moment. Let $F: [0,\infty)\to [0,1]$ be a $n$-times differentiable function such that $\int_0^{\infty} \zeta^{-\alpha}F^{(n)}(\zeta)\d \zeta$ is finite. Assume further that $\norm{F^{(k)}}_{\infty}<\infty$ for all $1\le k\le n$. Then the $(n-1+\alpha)$-th moment of $U$ is given by
		\begin{align*}
			\Ex [U^{n-1+\alpha}]=\dfrac{\int\limits_0^{\infty} \zeta^{-\alpha}\Ex[U^nF^{(n)}(\zeta U)]\d \zeta}{\int\limits_0^{\infty} \zeta^{-\alpha}F^{(n)}(\zeta)\d \zeta}=\dfrac{\int\limits_0^{\infty} \zeta^{-\alpha}\frac{\d^n}{\d\zeta^n}\Ex[F(\zeta U)]\d \zeta}{\int\limits_0^{\infty} \zeta^{-\alpha}F^{(n)}(\zeta)\d \zeta}.
		\end{align*}
	\end{lemma}
	The proof of this lemma follows by an interchange of measure justified by Fubini's theorem and the dominated convergence theorem, as $\Ex[U^n]$ and $\norm{F^{(k)}}_{\infty} < \infty$ for all $1 \le k \le n.$

	For $s>0$, we apply this lemma with $U = \tau^{H_0(t)}$, $n = \lfloor s \rfloor +1$ and $\alpha=s-\lfloor s \rfloor$. We take $F(x)= \fq(x)$ defined in \eqref{eq:laplace} which is shown to be satisfy the hypothesis of Lemma \ref{lm:frac_mom} (see Proposition \ref{p:etau}). As a result, we transform the computation of $\Ex[\tau^{sH_0(t)}]$ into that of \begin{align}\label{eq:int1}
		\int_0^{\infty} \zeta^{-\alpha}\frac{\d^n}{\d\zeta^n}\Ex[F_q(\zeta \tau^{H_0(t)})]\d \zeta.
	\end{align} 
	Utilizing the exact formula from \eqref{eq:laplace} and the definition of Fredholm determinant from \eqref{eq:f-series}, we can write the above expression as a series where we identify the leading term (corresponding to $L=1$ term of the series) and a higher-order term (corresponding to $L\ge 2$ terms of the series). We eventually show that the asymptotics of the leading term matches with the exact asymptotics in \eqref{eq:exp} while the higher-order term decays much faster. This leads to the proof of Theorem \ref{thm:frac_mom}.

	{The above description of our method is in line with the Lyapunov moment approach adopted in the works of \cite{dt19}, \cite{gl20} and \cite{lin20} to obtain upper-tail large deviation results of other integrable models, such as the KPZ equation. Namely, we extract fractional moments from the ($\tau$-)Laplace transform such as \eqref{eq:laplace} according to Lemma \ref{lm:frac_mom}. 
	In particular, our work draws from those of  \cite{dt19} and \cite{lin20}, which studied the fractional moments of the Stochastic Heat  Equation (SHE) and the half-line Stochastic Heat  Equation, respectively. We will further contextualize the connections of our work to \cite{dt19}, \cite{gl20} and \cite{lin20} in Section \ref{sec:pre}. In the following text, however, we emphasize a few key differences and technical challenges unique to the ASEP that we have encountered and resolved in our proof.}
	
	{First, unlike SHE or half-line SHE, the usual Laplace transform is not available in case of the ASEP. Instead, we only have the $\tau$-Laplace transform for our observable of interest. As a result, we have formulated Lemma \ref{lm:frac_mom} in our paper, which is more generalized than its prototype in \cite[Lemma 1.4]{dt19}, to feed in the $\tau$-Laplace transform. Consequently, we have worked with $\tau$-exponential functions in our analysis.}
	
	{Another key difference is that the kernel $K_{\zeta, t}$ in \eqref{def: ker} in our model is much more intricate than its counterpart in the KPZ model and leads to much more involved analysis of the leading term. Indeed, $K_{\zeta,t}$ is asymmetric and as $u$ varies in $(\delta - \i \infty, \delta + \i \infty)$, the function $\frac{g_t(w)}{g_t(\tau^uw)}$ appearing in the kernel $K_{\zeta,t}$, exhibits a periodic behavior, whereas the kernel in the KPZ models involves Airy functions in its integrand which have a unique maximum and are much easier to analyze. Furthermore, our model exhibits exponentially decaying moments  of $\tau^{H_0(t)}$ as opposed to the exponentially increasing ones of  the KPZ models in \cite{dt19} and \cite{lin20} and this demands a more precise understanding of the trace term of our Fredholm determinant expansion. For instance in Section \ref{sec:leading}, to obtain the precise asymptotics for our leading term, we have performed steepest descent analysis on the kernel $K_{\zeta, t}$, where the periodic nature of $\frac{g_t(w)}{g_t(\tau^uw)}$ results in infinitely many critical points. A major technical challenge in our proof is to argue how the contribution from only one of the critical points dominates the those from the rest and this is accomplished in the proof of Proposition \ref{p:leading}.
	Similarly, the asymmetry of the kernel in the ASEP model has led us to opt for the Hadamard's inequality approach as exemplified in Section 4 of \cite{lin20}, instead of the operator theory argument in \cite{dt19}, to obtain a sufficient upper bound for the higher-order terms in our paper in Section \ref{sec:higher}.}

  \subsection{Comparison to Previous Works}\label{sec:pre}
  	{In a broader context, our main result on the Lyapunov exponent for the ASEP with step initial data and its upper-tail large deviation belongs to the undertakings of studying the intermittency phenomenon and large deviation problems of  integrable models in the KPZ universality class.}
  As we have previously alluded to, the KPZ universality class contains a collection of random growth models that are characterized by scaling exponent of $1/3$ and certain universal  non-Gaussian large time fluctuations. We refer to \cite{acq,kpz,sasamoto} and the references therein for more details. The ASEP is one of the standard one-dimensional models of the KPZ universality class and bears connection to several other integrable models in this class, such as the stochastic six-vertex model \cite{bcg,agg,evgeni}, {KPZ equation \cite{cldr,dot10,ss10,acq,kpz}}, and $q$-TASEP \cite{bcs}.

On the other hand, the intermittency property is a universal phenomenon that captures high population concentrations on small spatial islands over large time. Mathematically, the intermittency of a random field is defined in terms of its Lyapunov exponents. In particular, the connection between integer Lyapunov moments and intermittency has long been an active area of study in the SPDE community in last few decades \cite{gar,car,ber,foon,hu,conus,chen,balan}. For the KPZ equation, \cite{kar} predicted the integer Lyapunov exponents for the SHE using replica Bethe anstaz techniques. {This result was later first rigorously attempted in \cite{ber} and correctly proven in \cite{che}}. Similar formulas were shown for the moments of the parabolic Anderson model, semi-discrete directed polymers, q-Whittaker process (see \cite{mcd} and \cite{bc}). 
For the ASEP, integer moments formula for $\tau^{H_0(t)}$ were obtained in \cite{bcs} using nested contour integral ansatz.

{From the perspective of tail events, by studying the asymptotics of  integer Lyapunov exponents formulas, one can extract one-sided bounds on the upper tails of integrable models. However, these integer Lyapunov exponents alone are not sufficient to provide the exact large deviation rate function.}

 Recently, a stream of effort has been devoted to studying large deviations for some KPZ class models by explicitly computing the fractional Lyapunov exponents.  
 The work of \cite{dt19} set this series of effort in motion by solving the KPZ upper-tail large deviation principle through the fractional Lyapunov exponents of the SHE with delta initial data. \cite{gl20} soon extended the same result for the SHE for a large class of initial data, including any {random} bounded positive initial data and the stationary initial data. An exact way to compute every positive Lyapunov exponent of the half-line SHE was also uncovered in \cite{lin20}. In lieu of these developments, our main result for the ASEP with step initial data and its upper-tail large deviation fits into this broader endeavor of studying large deviation problems of  integrable models with the Lyapunov exponent appproach.

	Meanwhile, in the direction of  the ASEP, as mentioned before, \cite{dp} has produced a one-sided large deviation bound for the upper-tail probability appearing in \eqref{eq:ldp}  which coincides with the correct rate function $\Phi_{+}$ defined in \eqref{eq:ldp} for $y \le y_0:= \frac{1-2\sqrt{q(1-q)}}{1+2\sqrt{q(1-q)}}$. {This result was sufficient for their purpose of establishing a near-exponential fixation time for the coarsening model on $\Z^2$ 
 and \cite{dp} obtained it via steepest descent analysis on the exact formula for the probability of $H_0(t/\gamma)$.} More specially, they worked with the following result from \cite[Lemma 4]{tw2} as input:
		\begin{equation}\label{eq:tw}
		\P\left(-H_0\big(\tfrac{t}\gamma\big)+\tfrac{t}{4} > \tfrac{t}{4}y\right)= \frac{1}{2\pi \i}\int_{|\mu|= R}(\mu; \tau)_{\infty}\det(1 + \mu J_{m, t}^{(\mu)})\frac{\d \mu}{\mu},
	\end{equation}
	where $m=\lfloor \frac{1}{4}t(1-y)\rfloor$, $R \in (\tau, \infty)\setminus\{1, \tau^{-1}, \tau^{-2}, \ldots\}$ is fixed, $(\mu;\tau)_{\infty}: = (1-\mu)(1 - \mu\tau)(1 - \mu\tau^2)\ldots$ is the infinite $\tau$-Pochhammer symbol and $J_{m,t}^{(\mu)}$ is the kernel defined in Equation (3.4) of \cite{dp}. Analyzing the exact pre-limit Fredholm determinant $\det(1 + \mu J_{m, t}^{(\mu)})$, \cite{dp} chose appropriate contours for the kernel $J_{m,t}^{(\mu)}$ that pass through its critical points and performed a steepest descent analysis. However, their choice of contours was unattainable beyond the threshold $y_0$. Namely, if we attempted to deform the same contours for  $y > y_0$, we would inevitably cross poles, which rendered the steepest descent analysis much trickier.  By adopting the Lyapunov moment approach, we have avoided this problem when looking for the precise large deviation rate function.
	 
	 In addition to the relavence of our upper-tail LDP result, it is also worthy to remark on the difficulty of obtaining a lower-tail LDP of the ASEP with step initial data.  As explained before, the lower-tail $\mathbb{P}(-H_0\big(\tfrac{t}\gamma\big)+\tfrac{t}{4} < -\tfrac{t}{4}y)$ is expected to go to zero at a much faster rate of $\exp(-t^2 \Phi_{-}(y))$. The existence of the lower-tail rate function has so far only been shown in the case of TASEP in \cite{joh} through its connection to continuous log-gases. {The functional LDPs for TASEP for both tails have been studied in \cite{jen}, \cite{var}, \cite{x3} (upper tail), and \cite{ot17} (lower-tail). Large deviations for open systems with boundaries in contact with stochastic reservoirs has also been studied in physics literature. We mention \cite{DL98}, \cite{DLS03}, \cite{BD06} and the references therein for works in these directions.}

	 {More broadly for integrable models in the KPZ universality class, lower tail of the KPZ equation has been extensively studied in both mathematics and physics communities. In the physics literature, \cite{le2016large} provided the first prediction of the large deviation tails of the KPZ equation for narrow wedge initial data. For the upper tail, their analysis also yields subdominant corrections (\cite[Supp.~Mat.]{supp}).  Furthermore, the physics work of \cite{sasorov2017large} first predicted lower-tail rate function of the KPZ equation for narrow wedge initial data  in an analytical form, followed by the derivations in \cite{JointLetter} and \cite{ProlhacKrajenbrink} via different methods. The asymptotics of deep lower tail of KPZ equation was later obtained in \cite{KrajLedou2018} for a wide class of initial data. From the mathematics front, the work \cite{cg} provided detailed, rigorous tail bounds for the lower tail of the KPZ equation for narrow wedge initial data. The precise rate function of its lower-tail LDP was later proved in \cite{tsai18} and \cite{cr}, which confirmed the prediction of existing physics literature. The four different routes of deriving the lower-tail LDP in \cite{sasorov2017large}, \cite{JointLetter}, \cite{ProlhacKrajenbrink} and \cite{tsai18} were later shown to be closely related in \cite{largedev}. A new route has also been recently obtained in the physics work of \cite{doussal2019large} (see also \cite{ProlhacRiemannKPZ}).}
	 	
	 	{In the short time regime, large deviations for the KPZ equation has been studied extensively in physics literature (see \cite{x1}, \cite{x2}, \cite{krajenbrink} and the references therein for a review). Recently, \cite{lt21} rigorously derived the large deviation rate function of the KPZ equation  in the short-time regime in a variational form and recovered deep lower-tail asymptotics, confirming existing physics predictions}.   For non-integrable models, large deviations of first-passage percolation were studied in \cite{chow} and more recently \cite{basu}. For last-passage percolation with general weights, recently, geometry of polymers under lower tail large deviation regime has been studied in \cite{basu2}.

	\subsection*{Notation} Throughout the rest of the paper, we use $\Con = \Con(a,b,c,\ldots) > 0$ to denote a generic deterministic positive finite
	constant that is dependent on the designated variables $a, b, c, \ldots$. However, its particular content may change from line to line. We also use the notation $\C(r)$ to denote a positively oriented circle with center at origin and radius $r>0$. 
	
	\subsection*{Outline} The rest of this article is organized as follows. In Section \ref{sec:fraccase}, we introduce the main ingredients for the proofs of Theorem \ref{thm:frac_mom} and \ref{thm:ldp}. In particular, we reduce the proof of our main results to Proposition \ref{p:leading} (asymptotics of the leading order) and Proposition \ref{p:ho} (estimates for the higher order), which are proved in Sections \ref{sec:leading} and \ref{sec:higher} respectively. Finally, in Appendix \ref{app} we compare our rate function $\Phi_{+}(y)$, defined in \eqref{eq:ldp}, to that of TASEP.
	
	\subsection*{Acknowledgements} We are grateful to Ivan Corwin for suggesting the problem and providing numerous stimulating discussions. His encouragement and inputs on earlier drafts of the paper have been invaluable. We also thank Evgeni Dimitrov, {Li-Cheng Tsai}, {Yier Lin} and Mark Rychnovsky for helpful conversations and {Pierre Le Doussal and Alexandre Krajenbrink for providing many valuable references to the physics literature}. The authors were partially supported by Ivan Corwin's NSF grant DMS:1811143 as well as the Fernholz Foundation's ``Summer Minerva Fellows" program.

	\section{Proof of Main Results} \label{sec:fraccase} 
	In this section, we give a detailed outline of the proofs of Theorems \ref{thm:frac_mom} and \ref{thm:ldp}. In Section \ref{sec:prop} we collect some useful properties of $\hq$ and $\fq$ functions defined in \eqref{eq:ldp} and \eqref{eq:laplace} respectively. In Section \ref{sec:proof} we complete the proof of Theorems \ref{thm:frac_mom} and \ref{thm:ldp} assuming technical estimates on the leading order term (Proposition \ref{p:leading}) and higher order term (Proposition \ref{p:ho}).

	Throughout this paper, we fix $s>0$ and set $n=\lfloor s\rfloor +1 \ge 1$ and $\alpha=s-\lfloor s\rfloor$ so that $s=n-1+\alpha$.  We also fix $q\in (\frac12,1)$ and set $p=1-q$ and $\tau=p/q \in (0,1)$ for the rest of the article.

	\subsection{Properties of $\hq(x)$ and $\fq(x)$} \label{sec:prop} Recall the Lyapunov exponent $\hq(x)$ defined in \eqref{eq:exp} and the $\fq(x)$ function defined in \eqref{eq:laplace}. The following two propositions investigates various properties of these two functions which are necessary for our later proofs. 

	\begin{proposition}[Properties of $\hq$] \label{p:htau} Consider the function $\hq: (0,\infty) \to \mathbb{R}$ defined by $\hq(x)=(q-p)\frac{1-\tau^{\frac{x}2}}{1+\tau^{\frac{x}2}}$. Then, the following properties hold true:
	\begin{enumerate}[label=(\alph*), leftmargin=15pt]
		\item \label{a} $\bq(x):=\frac{\hq(x)}{x}$ is strictly positive and strictly decreasing with $$\lim_{x\to 0^+} B_q(x)=\tfrac14(p-q)\log \tau>0.$$
		\item \label{b} $\hq$ is strictly subadditive in the sense that for any $x,y\in (0,\infty)$ we have $$\hq(x+y)<\hq(x)+\hq(y).$$
		\item \label{c} $\hq$ is related to $\Phi_{+}$ defined in \eqref{eq:ldp} via the following Legendre-Fenchel type transformation:
		\begin{align*}
			\Phi_{+}(y)=\sup_{s\in \R_{>0}} \left\{s\frac{1-y}{4}\log\tau+\frac1{q-p}\hq(s)\right\}, \quad y\in (0,1).
		\end{align*}
	\end{enumerate}
\end{proposition}

\begin{proof}
	For \ref{a}, first, the positivity of $\bq(x)$ follows from the positivity of $\hq(x).$ To see its growth, taking the derivative of $\bq(x)$ we obtain
	\begin{align}\label{eq:bq}
		\bq'(x)= \frac{(q-p)(-x\tau^{\frac{x}{2}}\log\tau -1 + \tau^x)}{(1 + \tau^{\frac{x}{2}})^2x^2}.
	\end{align}
	Note that the numerator on the r.h.s of (\ref{eq:bq}) is 0 when $x=0$ and its derivative against $x$ is $\tau^{\frac{x}{2}}\log \tau(\tau^{\frac{x}{2}}-\frac{x}{2}\log\tau -1) < 0$ for $x > 0$. Thus $\bq'(x)$ is strictly negative when $x> 0$ and $\bq(x)$ is strictly decreasing for $x>0$. L'H\^opital's rule yields that $\lim_{x\to 0^+}\bq(x) =\hq'(0)= \frac{1}{4}(q-p)\log\tau.$
	
	\medskip
	
	For \ref{b}, direct computation yields 
	\begin{align}\label{eq:diff}
		\hq(x+y)-\hq(x)-\hq(y)  =-(q-p)\frac{(1-\tau^{\frac{y}2})(1-\tau^{\frac{x}2})(1-\tau^{{\frac{x+y}2}})}{(1+\tau^{{\frac{x+y}2}})(1+\tau^{{\frac{x}2}})(1+\tau^{{\frac{y}2}})} < 0.
	\end{align}
	
	Lastly, for part \ref{c}, we fix $y\in (0,1)$ and define \begin{align*}
		g_{y}(s): = s\frac{1-y}{4}\log\tau+\frac1{q-p}\hq(s), \quad s>0.
	\end{align*}
	Direct computation yields $g_{y}'(s) = (\frac{1-y}{4} -\frac{\tau^{\frac{s}{2}}}{(1 + \tau^{\frac{s}{2}})^2})\log \tau$ and $g_y''(s)=\frac{\tau^{\frac{s}{2}}(\tau^{\frac{s}{2}}-1)\log^2\tau}{2(1+\tau^{\frac{s}{2}})^3} <0$. Thus $g_{y}(s)$ is concave on $(0,\infty)$ and hence attains its unique maxima when $g_y'(s)=0$ or equivalently $\frac{1-y}{4} =\frac{\tau^{\frac{s}{2}}}{(1 + \tau^{\frac{s}{2}})^2}.$ The last equation has $s = 2\log_{\tau} (\frac{1 - \sqrt{y}}{1 + \sqrt{y}})$ as the only positive solution and hence it defines the unique maximum. Substituting this $s$ back into $g_{y}(s)$ generates the final result as $\Phi_{+}(y).$
\end{proof}

	\begin{proposition}[Properties of $\fq(\zeta)$] \label{p:etau} Consider the function $\fq:[0,\infty) \to [0,1]$ defined by $\fq(\zeta):=\prod_{n=0}^{\infty}(1+\zeta\tau^n)^{-1}$. Then, the following properties hold true:
		\begin{enumerate}[label=(\alph*), leftmargin=15pt]
			\item \label{fa} $\fq$ is an infinitely differentiable function with $(-1)^n\fq^{(n)}(\zeta) \ge 0$ for all $x>0$. Furthermore, $\norm{\fq^{(n)}}_{\infty}<\infty$ for each $n$.
			\item \label{fb} For each $n\in \Z_{>0}$, and $\alpha\in [0,1)$, $(-1)^n\int_0^{\infty}\zeta^{-\alpha}\fq^{(n)}(\zeta)\d \zeta$ is  positive and finite.
			
			\item \label{fc} All the derivatives of $\fq$ have superpolynomial decay. In other words for any $m, n \in \Z_{\ge 0}$ we have
			$$\sup_{\zeta>0} |\zeta^m\fq^{(n)}(\zeta)| < \infty.$$
		\end{enumerate}
	\end{proposition}
	\begin{proof} 
	
			(a)
			Note that $\fq(\zeta)=\prod_{n=0}^{\infty}(1+\zeta\tau^n)^{-1}=(-\zeta;\tau)_{\infty}^{-1}$ where we recall that $(-\zeta;\tau)_{\infty}$ is the  $\tau$-Pochhammer symbol. As $(-\zeta;\tau)_{\infty}$ is analytic  \cite[Corollary A.1.6.]{aar} and nonzero for $\zeta \in [0, \infty),$ its inverse $\fq(\zeta)$ is analytic. 
			
			\smallskip
			
			We next rewrite $\fq(\zeta) = \prod_{n=0}^{\infty}f_n(\zeta),$ where $f_n(\zeta)=(1 + \zeta\tau^n)^{-1}$. Denote $H(\zeta): = \log \fq(\zeta).$ Since each $f_n(\zeta)\in(0,1)$ is analytic for $\zeta \in [0, \infty)$ and the product $\prod_{n=0}^{\infty}f_n(\zeta) \in (0,1) $ converges locally and uniformly,  $H(\zeta)$ is well-defined and  $H(\zeta) = \sum_{n=0}^{\infty}\log f_n(\zeta).$ Given that $|\sum_{n=0}^{\infty}\frac1{f_n(\zeta)}f_n'(\zeta)| = \sum_{n=0}^{\infty}\frac{\tau^n}{(1+\zeta\tau^n)}< \frac{1}{1-\tau},$ we have 
			\begin{equation}\label{eq:deqn}
				H'(\zeta) = \frac{ \fq'(\zeta)}{\fq(\zeta)} = \sum_{n=1}^{\infty}\frac{f_n'(\zeta)}{f_n(\zeta)}=: G(\zeta).
			\end{equation}
			Note that $G(\zeta) = -\sum_{j = 1}^{\infty}\tau^j f_j(\zeta)$ and $|G(\zeta)| < \infty.$ For each $m \in \Z_{> 0}$, let us set $G^{(m)}(\zeta) := -\sum_{j =1}^{\infty}\tau^jf_j^{(m)}(\zeta).$ As $f_j^{(m)}(\zeta) = (-1)^{m} m!\frac{\tau^{mj}}{(1+\xi\tau^j)^{m+1}},$ we obtain $|G^{(m)}(\zeta)| \leq \frac{m!}{1- \tau^{m+1}}< \infty$ converges locally and uniformly.  Induction on $m$ gives us that $G(\zeta)$ is infinitely differentiable and the $m$-th derivative of $G$ is $G^{(m)}$. It follows that  $\fq(\zeta)$ is infinitely differentiable too. In particular, for any finite $n \in \Z_{\ge0}$,  by Leibniz's rule on the relation \eqref{eq:deqn} we obtain \begin{align}\label{eq:leib}
			\fq^{(n+1)}(\zeta) = \sum_{k=0}^n\binom{n}{k}\fq^{(n-k)}(\zeta)G^{(k)}(\zeta).
			\end{align} Observe that $(-1)^{k+1}G^{(k)}$ is positive and finite. As $\fq$ is positive and finite, using \eqref{eq:leib}, induction  gives us that $(-1)^{n}\fq^{(n)}$ is also positive and finite. As $\norm{G^{(m)}}_{\infty}$ and $\norm{\fq}_{\infty}$ are finite, using \eqref{eq:leib}, induction gives us that $\norm{\fq^{(n)}}_{\infty}$ is finite for any $n \in \Z_{\ge0}.$
		
		\medskip
			
			(b) For $\alpha \in [0,1)$, positivity of the integral $(-1)^n\int_0^{\infty}\zeta^{-\alpha}\fq^{(n)}(\zeta)\d\zeta$ follows from part (a). To check the integrability, we first verify the $n = 0$ case. Since $\zeta \geq 0$ and $\tau \in (0,1),$
			\begin{equation*}
			\begin{split}
			0 &<\int_{0}^{\infty} \zeta^{- \alpha} \fq(\zeta) \d\zeta = \int_{0}^{\infty}\zeta^{- \alpha}\prod_{m=0}^{\infty}\frac{1}{1 + \zeta\tau^m}\d\zeta < \int_{0}^{\infty}\zeta^{-\alpha}\frac{1}{1 + \zeta}\d\zeta\\&= \int_{0}^1\zeta^{- \alpha}\frac{1}{1 + \zeta} \d\zeta + \int_1^{\infty}\frac{\d\zeta}{\zeta^{\alpha}(1 + \zeta)}<\int_0^{1} \zeta^{-\alpha} \d\zeta + \int_1^{\infty}\frac{\d\zeta}{\zeta^{\alpha + 1}}< \infty.
			\end{split}
			\end{equation*} When $n > 0$, using \eqref{eq:leib} and the fact the $|G^{(m)}(\zeta)|<\frac{m!}{1-\tau^{m+1}}$, the finiteness of $(-1)^n\int_0^{\infty}\zeta^{-\alpha}\fq^{(n)}(\zeta)\d\zeta$  follows from induction.
		
		\medskip
		
		(c) Clearly for each $m$ we have $\fq(\zeta) \le \frac{1}{(1+\zeta \tau^{m})^{m+1}}$ forcing superpolynomial decay of $\fq$. The superpolynomial decay of higher order derivative now follows via induction using \eqref{eq:leib}. 
	\end{proof}
	
	\subsection{Proof of Theorem \ref{thm:frac_mom} and Theorem \ref{thm:ldp}} \label{sec:proof}
	Recall $H_0(t)$ from \eqref{def:ht}. As explained in Section \ref{sec:ske}, the main idea is to use Lemma \ref{lm:frac_mom} with $U=\tau^{H_0(t)}$ and $F=\fq$ defined in \eqref{eq:laplace}. Observe that Proposition \ref{p:etau} guarantees $F=F_q$ can be chosen in Lemma \ref{lm:frac_mom}. In the following proposition, we show that limiting behavior of $\Ex [\tau^{sH_0(t)}]$ is governed by the  integral in \eqref{eq:int1} restricted to $[1,\infty)$.
	\begin{proposition}\label{p:red} For any $s>0$, we have
		\begin{align}\label{eq:red}
			\lim_{t\to\infty}\frac1t\log \Ex [\tau^{sH_0(t)}]=\lim_{t\to\infty}\frac1t\log \left[(-1)^n\int_1^{\infty} \zeta^{-\alpha}\frac{\d^n}{\d\zeta^n}\Ex[F_q(\zeta \tau^{H_0(t)})]\d \zeta\right],
		\end{align}
	where $n=\lfloor s\rfloor +1 \ge 1$ and $\alpha=s-\lfloor s\rfloor$ so that $s=n-1+\alpha$.
	\end{proposition}
	\begin{proof}
		Let $U=\tau^{H_0(t)}$.  In this proof, we find an upper and a lower bound of $\Ex[U^s]$ and show that as $t \rightarrow \infty,$ after taking logarithm of $\Ex[U^s]$ and dividing by $t$, the two bounds give matching results. Note that as $\tau \in (0,1)$ and $H_0(t) \geq 0$ for any $n\in \Z_{\geq 0}$ and $t > 0,$ $U$ has finite $n$-th moment. By Proposition \ref{p:etau}, $\fq$ is $n$-times differentiable and  $|\int_0^{\infty}x^{-\alpha}\fq^{(n)}(x)\d x| < \infty.$ Denoting $\d\P_U(u)$ as the measure corresponding to the random variable $U$ we have
		\begin{align}
			(-1)^n\int_1^{\infty} \zeta^{-\alpha}\frac{\d^n}{\d\zeta^n}\Ex[F_q(\zeta \tau^{H_0(t)})]\d \zeta & = (-1)^n\int_1^{\infty} \zeta^{-\alpha}\int_0^{\infty} u^n\fq^{(n)}(\zeta u)\d\P_U(u)\d \zeta. \label{eq:sim}
		\end{align}
		The $(-1)^n$ factor ensures that the above quantities are nonnegative via Proposition \ref{p:etau} \ref{fa}. By the finiteness of the $n$-th moment of $U$, $\norm{F_q^{(n)}}_{\infty}<\infty$ (by Proposition \ref{p:etau} \ref{fa}), and Fubini's theorem, we can interchange the integrals and obtain
		\begin{align}
			\mbox{r.h.s~of \eqref{eq:sim}}	& = (-1)^n\int_0^{\infty} u^{n-1+\alpha} \int_1^{\infty} (\zeta u)^{-\alpha} 
			\fq^{(n)}(\zeta u)\d (u\zeta) \d\P_U(u) \nonumber \\ & = (-1)^n\int_0^{\infty} u^{n-1+\alpha} \int_{u}^{\infty} x^{-\alpha} 
			\fq^{(n)}(x)\d x \ \d\P_U(u) \label{eq:sim2}
		\end{align}
		Since the random variable $U\in [0,1]$, we can lower bound the inner integral on the r.h.s.~of \eqref{eq:sim2} by restricting the $x$-integral to $[1,\infty)$. Recalling that $s=n-1+\alpha$ we have  \begin{align}\label{eq:low}
			\mbox{r.h.s.~of \eqref{eq:sim}} \ge (-1)^n\left(\int_1^{\infty} x^{-\alpha}F_q^{(n)}(x)\d x\right)\Ex [\tau^{sH_0(t)}].
		\end{align} 
		As for the upper bound for $\mbox{r.h.s.~of \eqref{eq:sim}}$, we may extend the range of integration to $[0,\infty)$. Apply Lemma \ref{lm:frac_mom} with $F\mapsto F_q$ and $U\mapsto \tau^{sH_0(t)}$ to get
		\begin{align}\label{eq:up}
			\mbox{r.h.s.~of \eqref{eq:sim}} \le (-1)^n\int_0^{\infty} \zeta^{-\alpha}\frac{\d^n}{\d\zeta^n}\Ex\left[
			F_q(\zeta U)\right]\d \zeta  = \left[(-1)^n\int_0^{\infty}\zeta^{-\alpha}\fq^{(n)}(\zeta)\d \zeta\right]\Ex [\tau^{sH_0(t)}] 
		\end{align}
		Noting that both the prefactors in \eqref{eq:low} and \eqref{eq:up} are positive and free of $t$. Taking logarithms and dividing by $t$, we get the desired result.
	\end{proof}

	Next we truncate the integral in r.h.s.~of \eqref{eq:red} further. Recall the function $\bq(x)$ defined in Proposition \ref{p:htau} \ref{a}. We separate the range of integration  $[1,\infty)$ into $[1,e^{t\bq(s/2)}]$ and $(e^{t\bq(s/2)},\infty)$ and make use of the Fredholm determinant formula for $\Ex[\fq(\zeta\tau^{H_0(t)})]$ from Theorem \ref{thm:laplace} to write the integral in r.h.s.~of \eqref{eq:red} as follows.  
	\begin{align} \nonumber
		(-1)^n\int_1^{\infty} \zeta^{-\alpha}\frac{\d^n}{\d\zeta^n}\Ex[F_q(\zeta \tau^{H_0(t)})]\d \zeta & =(-1)^n\int_1^{e^{t\bq(\frac{s}{2})}}  \zeta^{-\alpha}\frac{\d^n}{\d\zeta^n}\Ex[F_q(\zeta \tau^{H_0(t)})]\d \zeta+\mathcal{R}_s(t) \\ & = (-1)^n\int_{1}^{e^{t\bq(\frac{s}{2})}} \zeta^{-\alpha}\frac{\d^n}{\d\zeta^n}\det(I+K_{\zeta,t})\d \zeta +\calR_s(t), \label{eq:diff_fred}
	\end{align}
	where
	\begin{align}\label{eq:calr}
		\mathcal{R}_s(t):=(-1)^n\int_{e^{t\bq(\frac{s}{2})}}^{\infty} \zeta^{-\alpha}\frac{\d^n}{\d\zeta^n}\Ex[F_q(\zeta \tau^{H_0(t)})]\d \zeta
		\end{align}
Recall the definition of Fredholm determinant from \eqref{eq:f-series}. Assuming $\tr(K_{\zeta,t})$ to be differentiable for a moment we may split the first term in \eqref{eq:diff_fred} into two parts and write
\begin{align}\label{eq:sep2}
	(-1)^n\int_{1}^{e^{t\bq(\frac{s}{2})}} \zeta^{-\alpha}\frac{\d^n}{\d\zeta^n}\det(I+K_{\zeta,t})\d \zeta = \calA_s(t)+\calB_s(t)
\end{align}
	where
	\begin{align}\label{eq:cala}
		\calA_{s}(t) & := (-1)^n\int_{1}^{e^{t\bq(\frac{s}{2})}} \zeta^{-\alpha}\frac{\d^n}{\d\zeta^n}\tr( K_{\zeta,t})\,\d\zeta, \\
		\calB_{s}(t) & := (-1)^n\int_{1}^{e^{t\bq(\frac{s}{2})}} \zeta^{-\alpha}\frac{\d^n}{\d\zeta^n}[\det(I+K_{\zeta,t})-\tr( K_{\zeta,t})]\,\d\zeta. \label{eq:Calb}
	\end{align}
	
	The next two propositions verify that both $\mathcal{A}_s(t)$ and $\mathcal{B}_{s}(t)$ are well-defined and we defer their proofs to Sections \ref{sec:leading} and \ref{sec:higher}, respectively. The first one guarantees that $\tr(K_{\zeta,t})$ is indeed infinitely differentiable and provides the asymptotics for $\rl [\calA_s(t)]$.
	\begin{proposition}\label{p:leading}  For each $\zeta>0$, the function $\zeta \mapsto\tr(K_{\zeta,t})$ is infinitely differentiable and thus $\calA_s(t)$ in \eqref{eq:cala} is well defined. Furthermore, for any $s>0$, we have
		\begin{align}\label{eq:lead}
			\lim_{t\to\infty} \log\left( \rl[\calA_{s}(t)]\right) = -\hq(s).
		\end{align}
	\end{proposition}
	
	From \eqref{eq:diff_fred}, we know that the Fredholm determinant $\det(I + K_{\zeta, t})$ is infinitely differentiable. Thus, proposition \ref{p:leading} renders $(\det(I +K_{\zeta,t})-\tr(K_{\zeta,t}))$ infinitely differentiable as well. Hence $\mathcal{B}_s(t)$ is well-defined. In fact, we have the following asymptotics for $\mathcal{B}_{s}(t)$.
	\begin{proposition} \label{p:ho} Fix any $s>0$ so that $s-\lfloor s\rfloor>0$. Recall $\calB_s(t)$ from \eqref{eq:Calb}. There exists a constant $\Con=\Con(q,s)>0$ such that for all $t>0$, we have
		\begin{align}\label{eq:ho}
			|\mathcal{B}_s(t)| \le \Con \exp(-t\hq(s)-\tfrac1{\Con}t),
		\end{align}
		where $\hq(s)$ is defined in \eqref{eq:exp}.
	\end{proposition}
Note that Proposition \ref{p:ho} in its current form does not cover integer $s$. We later explain in Section \ref{sec:higher} why $s-\lfloor s \rfloor >0$ is necessary for our proof. However, this does not effect our main results as one can deduce Theorem \ref{thm:frac_mom} for integer $s$ as well via a simple continuity argument, which we present below. Assuming Propositions \ref{p:leading} and \ref{p:ho}, we now complete the proof of Theorem \ref{thm:frac_mom} and Theorem \ref{thm:ldp}.
	
	\begin{proof}[Proof of Theorem \ref{thm:frac_mom}] Fix $s>0$ so that $s-\lfloor s \rfloor >0$. Appealing to  Proposition \ref{p:red} and \eqref{eq:diff_fred} and \eqref{eq:sep2} we see that
		$$\lim_{t\to\infty} \frac1t\log\Ex [\tau^{sH_0(t)}] =\lim_{t\to\infty} \frac1t\log\left[\mathcal{A}_s(t)+\mathcal{B}_s(t)+\mathcal{R}_s(t)\right],$$
		where $\mathcal{A}_s(t)$, $\mathcal{B}_s(t)$,  and $\mathcal{R}_s(t)$ are defined in \eqref{eq:cala}, \eqref{eq:Calb} and \eqref{eq:calr} respectively. For $\mathcal{R}_s(t)$, setting $V=\zeta \tau^{H_0(t)}$ and noting $s = n-1+ \alpha,$ we see that
		\begin{align*}
			|\mathcal{R}_s(t)| = \int_{e^{t\bq(\frac{s}{2})}}^{\infty} \zeta^{-\alpha-n}\Ex\left[|V^n\fq^{(n)}\left(V\right)|\right]\d \zeta \le \left[\sup_{v>0} |v^{n}\fq^{(n)}(v)|\right] s^{-1}\exp(-ts\bq(\tfrac{s}2)).
		\end{align*}
		The fact that $\sup_{v>0} |v^{n}\fq^{(n)}(v)|$ is finite follows from Proposition \ref{p:etau} \ref{fc}.
		Note that $s\bq(\tfrac{s}2)$ is strictly bigger than $\hq(s)=s\bq(s) > 0$ via Proposition \ref{p:htau} \ref{a}. By Proposition \ref{p:leading}, when $t$ is large, we see that  $\rl[\calA_{s}(t)]$ grows like $\exp(-t\hq(s)) > \exp(-ts\bq(\frac{s}{2}))$. Similarly, Proposition \ref{p:ho} shows that $\rl[\calB_s(t)]$ is bounded from above by $\Con\exp(-t\hq(s) -\frac{1}{\Con}t)$ for some constant $\Con = \Con(q,s)$, which is strictly less than $\exp(-t\hq(s))$ for large enough $t$. Indeed for all large enough $t$, we have $$\frac12\rl[\calA_{s}(t)]\le \rl[\calA_{s}(t)+\calB_s(t)+\mathcal{R}_s(t)] \le\frac32\rl[\calA_{s}(t)].$$ Taking logarithms and dividing by $t$, and noting that $\calA_{s}(t)+\calB_s(t)+\mathcal{R}_s(t)$ is always real, we get \eqref{eq:exp} for any noninteger positive $s$. 
		
		To prove \eqref{eq:exp} for positive integer $s$, we fix $s\in \Z_{>0}$. For any $K>2$, observe that as $H_0(t)$ is a non-negative random variable (recall the definition from \eqref{def:ht}) we have
		$$\tau^{(s-K^{-1})H_0(t)} \ge \tau^{sH_0(t)} \ge \tau^{(s+K^{-1})H_0(t)}.$$ 
		 Taking expectations, then logarithms and dividing by $t$, in view of noninteger version of \eqref{eq:exp} we have
		 $$-\hq(s-K^{-1}) \ge \limsup_{t\to\infty} \frac1t\log\Ex[\tau^{sH_0(t)}] \ge \liminf_{t\to\infty} \frac1t\log\Ex[\tau^{sH_0(t)}] \ge -\hq(s+K^{-1}).$$
		Taking $K\to \infty$ we get the desired result for integer $s$.
		\end{proof}
		\begin{proof}[Proof of Theorem \ref{thm:ldp}]
		For the large deviation result, applying Proposition 1.12 in \cite{gl20}, with $X(t)=H_0(t/\gamma)\cdot\log\tau$, and noting the Legendre-Fenchel type identity for $\Phi_{+}(y)$ from Proposition \ref{p:htau} \ref{c}, we arrive at \eqref{eq:ldp}. To prove \eqref{eq:asy}, applying L-H\^opital rule a couple of times we get
		$$\lim_{y\to 0^+}\frac{\Phi_{+}(y)}{y^{3/2}}= \lim_{y\to0^+} \frac{2}{3}\frac{\Phi_{+}'(y)}{\sqrt{y}}= \lim_{x\to0^+} \frac{2}{3}\frac{\tanh^{-1}(x)}{x}=\lim_{x\to 0^+}\frac23\cdot\frac{1}{1-x^2}=\frac{2}{3}.$$
		This completes the proof of the theorem.
	\end{proof}
	
	\section{Asymptotics of the Leading Term} \label{sec:leading}

The goal of this section is to obtain exact asymptotics of $\rl[\mathcal{A}_s(t)]$ defined in \eqref{eq:cala} as $t\to\infty$. Recall the definition of the kernel $K_{\zeta,t}$ from \eqref{def: ker}. We employ a standard idea that the asymptotic behavior of the kernel $K_{\zeta,t}$ and its `derivative' (see \eqref{def: kerdrv}) and subsequently that of $\rl[\mathcal{A}_s(t)]$ can be derived by the \textit{steepest descent method}. 

Towards this end, we first collect all the technical estimates related to the kernel $K_{\zeta,t}$ in Section \ref{sec:tech} and go on to complete the proof of Proposition \ref{p:leading} in Section \ref{sec:lead}. 

\subsection{Technical estimates of the Kernel} \label{sec:tech}

In this section, we analyze the kernel $K_{\zeta,t}$. Much of our subsequent analysis boils down to understanding the function $g_t(z)$, defined in \eqref{def: ker}, that appears in the kernel $K_{\zeta,t}$. Towards this end, we consider 
\begin{align}\label{eq:contour_fn}
	f(u,z):=\frac{(q-p)}{1+\frac{z}{\tau}}-\frac{(q-p)}{1+\frac{\tau^uz}{\tau}},
\end{align}	
so that the ratio $\frac{g_t(z)}{g_t(\tau^uz)}$ that appears in the kernel $K_{\zeta,t}$ defined in \eqref{def: ker} equals  to $\exp\left(tf(u,z)\right)$. Below we collect some useful properties of this function $f(u,z)$. First note that  $\partial_{z}f(u,z)=0$ has two solutions $z=\pm \tau^{1-\frac{u}2}$, and 
\begin{align}\label{eq;deri}
	\partial_{z}^2f(u,z)\big\vert_{z =-\tau^{1-\frac{u}2}}=-2(q-p)\frac{\tau^{\frac{3u}2-2}+\tau^{2u-2}}{(1-\tau^{\frac{u}2})^3}, \ \ 
	\partial_{z}^2f(u,z)\big\vert_{z =\tau^{1-\frac{u}2}}=2(q-p)\frac{\tau^{\frac{3u}2-2}-\tau^{2u-2}}{(1+\tau^{\frac{u}2})^3}.
\end{align}
The following lemma tells us how the maximum of $\rl [f(u,z)]$ behaves. 
\begin{lemma}\label{lem:max} Fix $\rho>0$. For any $u\in \mathbb{C}$, with $\rl [u]=\rho$ and $z \in \C(\tau^{1-\frac{\rho}2})$, we have \begin{align} \label{eq:ineq}
		\rl [f(u,z)]\le f(\rho,\tau^{1-\frac{\rho}2})=-\hq(\rho)
	\end{align}  where $\hq(\rho)$ is defined in \eqref{eq:exp} and $\C(\tau^{1-\frac{\rho}2})$ is the circle with center at the origin and radius $\tau^{^{1-\frac{\rho}2}}$. Equality in \eqref{eq:ineq} holds if and only if $\tau^{\i \im u}=1$, and $z=\tau^{1-\frac{\rho}2}$ simultaneously.  Furthermore, for the same range of $u$ and $z$, we have the following inequality:
	\begin{align}\label{max:ineq}
		f(\g,\tau^{1-\frac{\rho}2})-\rl [f(u,z)] \ge \frac{(q-p)(1-\tau^{\frac\g2})\tau^{\frac{\rho}2}}{4(1+\tau^{\frac{\rho}2})^2}(2\tau^{\frac{\rho}2-1}|z-\tau^{1-\frac{\rho}2}|+|\tau^{\i \im u}-1|).
	\end{align}
\end{lemma}
\begin{proof} Set $u=\rho+\i y$ and $z=\tau^{1-\frac{\rho}2}e^{\i\theta}$ with $x\in \R$ and $\theta\in [0,2\pi]$. Note that $f(\rho,\tau^{1-\frac{\rho}2})=-\hq(\rho)$, where $\hq(x)$ is defined in \eqref{eq:exp}. Direct computation yields
	\begin{align}\label{eq:rl}
		\rl[f(u,z)]= \frac{(q-p)(\tau^{\g}-1)(|1+\tau^{\frac\g2}e^{-\i\theta}|^2+|1+\tau^{\frac\g2+\i y}e^{\i\theta}|^2)}{2|1+\tau^{\frac\g2}e^{-\i\theta}|^2|1+\tau^{\frac\g2+\i y}e^{\i\theta}|^2}.
	\end{align}
	Since $\tau<1$, applying the inequality $|1+\tau^{\frac\g2}e^{-\i\theta}|^2+|1+\tau^{\frac\g2+\i y}e^{\i\theta}|^2 \ge 2|1+\tau^{\frac\g2}e^{-\i\theta}||1+\tau^{\frac\g2+\i y}e^{\i\theta}|,$ and then noting that  $|1+\tau^{\frac\g2}e^{-\i\theta}||1+\tau^{\frac\g2+\i y}e^{\i\theta}|\le (1+\tau^{\frac\g2})^2$, we see $(\mbox{r.h.s.~of \eqref{eq:rl}}) \le -(q-p)\frac{1-\tau^{\frac{\rho}2}}{1+\tau^{\frac{\rho}2}}$. Clearly equality holds if and only if $\theta=0$ and $\tau^{\i y}=1$ simultaneously.  Furthermore, following the above inequalities, we have $\rl[f(\rho+\i y,z)] \le -(q-p)\frac{1-\tau^{\frac\rho2}}{|1+\tau^{\frac\rho2}e^{\i\theta}|}$ and $\rl[f(\rho+\i y,z)] \le -(q-p)\frac{1-\tau^{\frac\rho2}}{|1+\tau^{\frac\rho2+\i y}e^{\i\theta}|}$. This yields
	\begin{equation}\label{eq:fr}
		\begin{aligned}
			f(\rho,\tau^{1-\frac{\rho}2})-\rl [f(\rho+\i y,z)] & \ge (q-p)\left[\frac{1-\tau^{\frac{\rho}2}}{|1+\tau^{\frac\rho2}e^{\i\theta}|}-\frac{1-\tau^{\frac{\rho}2}}{1+\tau^{\frac{\rho}2}}\right] \hspace{-0.07cm} \ge \hspace{-0.07cm} \frac{(q-p)(\tau^{\frac{\rho}2}-\tau^\rho)|e^{\i\theta}-1|}{(1+\tau^{\frac{\rho}2})^2}
		\end{aligned}
	\end{equation}
	and
	\begin{align*} 
		f(\rho,\tau^{1-\frac\rho2})-\rl [f(\rho+\i y,z)] & \ge (q-p)\left[\frac{1-\tau^{\frac{\rho}2}}{|1+\tau^{\frac\rho2+\i y}e^{\i\theta}|}-\frac{1-\tau^{\frac{\rho}2}}{1+\tau^{\frac{\rho}2}}\right]   \ge \frac{(q-p)(1-\tau^{\frac\g2})\tau^{\frac{\rho}2}|\tau^{\i y}e^{\i\theta}-1|}{(1+\tau^{\frac{\rho}2})^2}.		
	\end{align*}
	Adding the above two inequalities we have $f(\rho,\tau^{1-\frac\rho2})-\rl [f(\rho+\i y,z)] \ge \frac{(q-p)(1-\tau^{\frac\g2})\tau^{\frac\rho2}|\tau^{\i y}-1|}{2(1+\tau^{\frac\rho2})^2}$. Combining this with \eqref{eq:fr} and the substitution $ \tau^{1-\frac{\rho}{2}}e^{\i \theta}=z$ we get \eqref{max:ineq}.
	This completes the proof.
\end{proof}

Using the above technical lemma we can now explain the proof of Theorem \ref{thm:laplace}.

\begin{proof}[Proof of Theorem \ref{thm:laplace}] Due to Theorem 5.3 in \cite{bcs}, the only thing that we need to verify is  
	\begin{align} \label{eq:cond}
		\inf_{\substack{w,w'\in \C({\tau^{1-\frac{\delta}2}})\\ u\in \delta+\i\R}} |w'-\tau^uw|>0\quad \mbox{and} \quad \sup_{\substack{w,w'\in \C({\tau^{1-\frac{\delta}2}})\\ u\in \delta+\i\R}}  \left|\frac{g_t(w)}{g_t(\tau^uw)}\right|>0.
	\end{align}
	Indeed, for every $u\in \delta+\i \R$ and $w,w'\in \C(\tau^{1-\frac{\delta}2})$, we have $|w'-\tau^uw| \ge |w'|-|\tau^uw|=\tau^{1-\frac{\delta}2}-\tau^{1+\frac{\delta}2}>0$.  Recall $f(u,z)$ from \eqref{eq:contour_fn}.  Applying Lemma \ref{lem:max} with $\rho\mapsto \delta$ yields $$\left|\frac{g_t(w)}{g_t(\tau^uw)}\right|=|\exp(tf(u,w))| =\exp(t\rl [f(u,w)]) \le \exp(t f(\delta,\tau^{1-\frac\delta2}))=\exp(-t\hq(\delta)),$$
	where $\hq$ is defined in \eqref{eq:exp}. This verifies \eqref{eq:cond} and completes the proof. 
\end{proof}

\begin{remark} We now explain our choice of the contour $K_{\zeta,t}$ defined in \eqref{def: ker}, which comes from the method of steepest descent. Suppose $\rl[u]=\delta$. As noted before, directly taking derivative of  $f(u,z)= \exp(\frac{g_t(z)}{g_t(\tau^u z)})$, with respect to $z$ suggests that critical points are at $z = \pm\tau^{1 -\frac{u}{2}}$, and thus we take our contour to be $\C(\tau^{1-\frac{\delta}{2}}),$ so that it passes through the critical points. 
\end{remark}
Next we turn to the case of differentiability of $\tr(K_{\zeta,t})$ where  $K_{\zeta, t}$ is defined in (\ref{def: ker}). Using the function $f$ defined in \eqref{eq:contour_fn}, we rewrite the kernel as follows.
$$K_{\zeta,t}(w, w')= \frac{1}{2\pi \i}\int_{\delta-\i \infty}^{\delta + \i \infty}\Gamma(-u)\Gamma(1+u)\zeta^u e^{tf(u,w)}\frac{\d u}{w'-\tau^u w}.$$

Differentiating the integrand inside the integral in $K_{\zeta, t}(w. w')$ $n$-times defines a sequence of kernel $\{K_{\zeta, t}^{(n)}\}_{n \ge 1}: L^2(\C(\tau^{1-\frac{\delta}{2}})) \rightarrow L^2(\C(\tau^{1-\frac{\delta}{2}}))$ given by the kernel:
\begin{align}\label{def: kerdrv}
	K_{\zeta,t}^{(n)}(w,w') := \frac1{2\pi\i}\int_{\delta-\i\infty}^{\delta+\i\infty}\Gamma(-u)\Gamma(1+u)(u)_n\zeta^{u-n} e^{t f(u,w)}\frac{\d u}{w'-\tau^uw},
\end{align} where $(a)_n:=\prod_{i=0}^{n-1}(a-i)$ for $n \in \Z_{>0}$ and $(a)_0 = 1$ is the Pochhammmer symbol and $\delta\in (0,1)$. We also set $K_{\zeta,t}^{(0)}:=K_{\zeta,t}$. 
\begin{remark}  We remark that unlike Lemma 3.1 in \cite{dt19}, we do not aim to show that $K_{\zeta,t}$ is differentiable as an operator, or its higher order derivatives are equal to the operator $K_{\zeta,t}^{(n)}$. Indeed, showing convergence in the trace class norm is more involved because of the lack of symmetry and positivity of the operator $K_{\zeta,t}$. However, since we are dealing with the Fredholm determinant series only, for our analysis it is enough to investigate how each term of the series are differentiable and how their derivatives are related to $K_{\zeta,t}^{(n)}$.
\end{remark}
\begin{remark}\label{r:pole}
	Note that when viewing $K_{\zeta,t}^{(n)}$ as a complex integral, we can  deform its $u$-contour to $\g+\i\R$ for any $\g \in(0, n\vee 1)$. This is due to the analytic continuity of the integrand as the factor $(u)_n$ removes the poles at $ 1, \ldots, n-1$ of $\Gamma(-u).$
\end{remark}
The following lemma provides estimates of $K^{(n)}_{\zeta, t}$ that is useful for the subsequent analysis in Sections \ref{sec:leading} and \ref{sec:higher}.
\begin{lemma} \label{l:kdbd}
	Fix  $n \in \Z_{\ge 0}, t>0,\delta, \rho \in (0,n\vee 1),$ and consider any borel set $A\subset \R$. Recall $h_q(x)$ and $\bq(x)$ from Proposition \ref{p:htau} and $K_{\zeta,t}^{(n)}$ from \eqref{def: kerdrv}. For any $w\in \C(\tau^{1-\frac{\delta}{2}})$ and $w'\in \mathbb{C}$ and $\zeta \in [1,e^{t\bq(\frac{s}2)}]$, there exists a constant $\Con = \Con(n, \delta, q)>0$ such that whenever $|w'|\neq \tau^{1+\frac\delta2}$ we have
	\begin{equation}\label{eq: idbd}
		\begin{aligned}
			\int_{A}\left|\frac{(\delta +\i y)_n \zeta^{\rho -n +\i y}}{\sin(-\pi(\delta + \i y))}e^{tf({\delta+\i y}, w)}\right|\frac{\d y }{|w'-\tau^{\delta+\i y}w|} &
			\leq \frac{\Con\zeta^{\rho - n }}{||w'|-\tau^{1+\frac\delta2}|}\exp(t\cdot \sup_{y\in A} { \rl [f({\delta+\i y}, w)]}) \\ & \leq \frac{\Con\zeta^{\rho - n }}{||w'|-\tau^{1+\frac\delta2}|}\exp(-t\hq(\delta)).
		\end{aligned}.
	\end{equation}
	In particular when $w'\in \C(\tau^{1-\frac{\delta}2})$ we have
	\begin{equation} \label{eq: kdbd}
		|K_{\zeta, t}^{(n)}(w, w')|
		\leq \Con \zeta^{\delta - n }\exp(-t \hq(\delta)).
	\end{equation}
	Consequently, $K_{\zeta,t}^{(n)}(w, w')$ is continuous  in the $\zeta$-variable.
\end{lemma}
\begin{proof}
	Fix $n\in \Z_{\ge 0}, t>0,$ $\delta,\rho\in (0,n\vee 1)$ and $w \in \mathfrak{C}(\tau^{1- \frac{\delta}{2}})$ and $w'\in \mathbb{C}$ such that $|w'|\neq \tau^{1+\frac\delta2}$. Throughout the proof the constant $\Con>0$ depends on $n,\delta,$ and $q$ -- we will not mention it further.
	
	\smallskip
	
	Consider the integral on the r.h.s.~of \eqref{eq: idbd}. Observe that when $\delta\notin \Z$, $|(\delta + \i y)_n|\leq \Con|y|^n$  and $\frac{1}{|\sin(-\pi(\delta + \i y))|} \leq \Con e^{-|y|/\Con}$.  For $n\ge 2$, and $\delta \in \Z_{>0}\cap (0,n)$, we observe that the product $(\delta +\i y)_n$ contains the term $\i y$. Hence $|\frac{\i y}{\sin(-\pi(\delta +\i y))}| = |\frac{\i y}{\sin(-\pi(\i y))}| \le \Con e^{-|y|/\Con}$ for such an integer $\delta$. Whereas, $|\frac{\delta+\i y}{\i y}| \le \Con |y|^{n-1}$ for such an integer $\delta$. Finally, $|w'-\tau^{\delta+\i y}w| \geq  ||w'| - |\tau^{\delta}w||=||w'|-\tau^{1+\frac\delta2}|$. Combining the aforementioned estimates, we obtain that
	\begin{align*}
		\mbox{r.h.s. of }(\ref{eq: idbd})\leq \int_{A}\Con |y|^n e^{- |y|/\Con}\zeta^{\rho - n}|e^{tf({\delta+\i y}, w)}|\frac{\d y}{||w'| - \tau^{1+\frac\delta2}|}.
	\end{align*}
	Since $\int_{\R}|y|^ne^{-|y|/\Con}\d y$ converges applying $|e^{tf({\delta+\i y}, w)}| \le  e^{t\rl [f({\delta+\i y}, w)]}$ we arrive at the first inequality in \eqref{eq: idbd}. The second inequality follows by observing $\rl [f({\delta+\i y},w)] \le -\hq(\delta)$ by Lemma \ref{lem:max}.

	Recall $K_{\zeta,t}^{(n)}$ from \eqref{def: kerdrv}. Recall from Remark \ref{r:pole} that the $\delta$ appearing in \eqref{def: kerdrv} can be chosen in $(0, n\vee 1)$. Pushing the absolute value sign inside the explicit formula in \eqref{def: kerdrv} and applying Euler's reflection principle with change of variables $u = \delta + \i y$  yield
	\begin{equation*}
		|K_{\zeta,t}^{(n)}(w,w')| \le \frac{1}{2\pi}\int_{\R}\left|\frac{(\delta +\i y)_n \zeta^{\delta -n +\i y}}{\sin(-\pi(\delta + \i y))}e^{tf({\delta+\i y}, w)}\right|\frac{\d y }{|w' - \tau^{{\delta}+ \i y}w|} .
	\end{equation*}
	\eqref{eq: kdbd} now follows from \eqref{eq: idbd} by taking $\rho=\delta$. To see the continuity of $K_{\zeta,t}^{(n)}(w, w')$ in $\zeta,$ we fix $\zeta_1< \zeta_2 < \zeta_1 +1.$ By repeating the same set of arguments as above we arrive at 
	\begin{align}\label{eq: kcz}
		|K_{\zeta_2,t}^{(n)}(w, w') - K_{\zeta_1,t}^{(n)}(w, w')|\le C|\zeta_2^{\delta -n} - \zeta_1^{\delta -n}|\exp(-t\hq(\delta))
	\end{align}
	with the same constant $\Con$ in \eqref{eq: kdbd}. Clearly l.h.s.~of \eqref{eq: kcz} converges to 0 when $\zeta_2 \rightarrow \zeta_1$, which confirms the kernel's $\zeta$-continuity.
\end{proof}

\subsection{Proof of Proposition \ref{p:leading}} \label{sec:lead} The goal of this section is to prove Proposition \ref{p:leading}. Before diving into the proof, we first settle the infinite differentiability separately in the next proposition.

\begin{proposition}\label{ppn:dkernel} For any $n\in \Z_{\ge0}$ and $t>0$, the operator $K_{\zeta,t}^{(n)}$ defined in \eqref{def: kerdrv} is a trace-class operator with
	\begin{align}\label{eq:inttr}
		\tr (K_{\zeta,t}^{(n)}) = \frac1{2\pi \i}\int_{\C(\tau^{1-\frac\delta2})} K_{\zeta,t}^{(n)}(w,w)\d w.
	\end{align}
	Furthermore, $\tr (K_{\zeta,t}^{(n)})$ is differentiable in $\zeta$ at each $\zeta>0$ and we have $\partial_{\zeta} \tr (K_{\zeta,t}^{(n)}) = \tr (K_{\zeta,t}^{(n+1)})$. 
\end{proposition}
\begin{proof} Fix $n\in \Z_{\ge 0}, t>0$, and $\zeta>0$. $K_{\zeta,t}^{(n)}(w,w')$ is simultaneously continuous in both $w$ and $w'$ and $\partial_{w'}K_{\zeta,t}^{(n)}(w,w')$ is continuous in $w'$. By Lemma 3.2.7 in \cite{mcd} (also see \cite[page 345]{lax} or \cite{bor}) we see that $K_{\zeta,t}^{(n)}$ is indeed trace-class, and thus \eqref{eq:inttr} follows from Theorem 12 in \cite[Chapter 30]{lax}. 
	To show differentiability of $\tr (K_{\zeta,t}^{(n)})$ in variable $\zeta$, we fix $\zeta_1,\zeta_2>0$. Without loss of generality we may assume $\zeta_1+1>\zeta_2>\zeta_1$. Let us define
	\begin{equation*}
		\begin{split}
			D_{\zeta_1,\zeta_2} & := \frac{\tr (K_{\zeta_2,t}^{(n)}) - \tr (K_{\zeta_1,t}^{(n)})}{\zeta_2-\zeta_1}- \tr (K_{\zeta_1,t}^{(n+1)}) \\ & = \frac1{(2\pi\i)^2}\int_{\C(\tau^{1-\frac{\delta}{2}})}\int_{\delta-\i \infty}^{\delta+\i \infty} \Gamma(-u)\Gamma(1+u)R_{\zeta_1,\zeta_2;n}(u)e^{tf(u,w)}\frac{\d u}{w-\tau^u w}  \d w,
		\end{split}
	\end{equation*}
	where
	\begin{align}\label{eq:rdef}
		R_{\zeta_1,\zeta_2;n}(u):=(u)_n\left[\frac{\zeta_2^{u-n}-\zeta_1^{u-n}}{\zeta_2-\zeta_1} - (u-n)\zeta_1^{u-n-1}\right]= \int_{\zeta_1}^{\zeta_2} \frac{(\zeta_2-\sigma)}{{\zeta_2-\zeta_1}}(u)_{n+2}\sigma^{u-n-2}\d\sigma
	\end{align}
	Taking absolute value and appealing to Euler's reflection principle, we obtain 
	\begin{align}\label{eq: d}
		|D_{\zeta_1,\zeta_2}| & \le \left|\frac1{(2\pi\i)^2}\int_{\C(\tau^{1-\frac{\delta}{2}})}\int_{\delta-\i \infty}^{\delta+\i \infty}\int_{\zeta_1}^{\zeta_2} \frac{(u)_{n+2}}{\sin(-\pi u)} \frac{(\zeta_2-\sigma)}{{\zeta_2-\zeta_1}}\sigma^{u-n-2} e^{tf(u,w)}\frac{\d\sigma\d u}{w-\tau^u w}  \d w\right| \\ & \le \frac{\tau^{1-\frac{\delta}{2}}}{2\pi}\int_{\zeta_1}^{\zeta_2}|\sigma^{\delta+\i y-n-2}|\d\sigma \cdot \max_{w\in \C(\tau^{1-\frac\delta2})}\int_{\R} \frac{(\delta+\i y)_{n+2}}{\sin(-\pi(\delta+\i y))}  |e^{tf(\delta+\i y,w)}|\frac{\d y}{|w-\tau^{\delta+\i y} w|}. \nonumber
	\end{align}
	Note that Lemma \ref{l:kdbd} (\eqref{eq: idbd} specifically) we see that the above maximum is bounded by $\Con \exp(-t\hq(\delta))$ where the constant $\Con$ is same as in \eqref{eq: idbd}.   Since $|\sigma^{u-n-2}|= |\sigma^{\delta-n-2}|\le |\zeta_1^{\delta-n-2}|$ over the interval $[\zeta_1, \zeta_2]$ for $\delta \in (0,n\vee 1)$, we obtain 
	\begin{equation*}
		|D_{\zeta_1,\zeta_2}| \le  \Con \exp(-\hq(\delta))\int_{\zeta_1}^{\zeta_2}|\sigma^{u-n-2}|\d \sigma \leq \Con \exp(-t\hq(\delta)) (\zeta_2 -\zeta_1)|\zeta_1^{\delta-n-2}|.
	\end{equation*}
	Thus, taking the limit as $ \zeta_2 - \zeta_1 \rightarrow 0$ yields $|D_{\zeta_1, \zeta_2}|\rightarrow 0$ and completes the proof.
\end{proof}

\begin{remark} We prove a higher order version of Proposition \ref{ppn:dkernel} later in Section \ref{sec:higher} as Proposition \ref{p:trL} which includes the statement of the above Proposition when $L =1$. However, we keep the above simple version for reader's convenience, which will serve as a guide in proving Proposition \ref{p:trL}.
\end{remark}

With the above results in place, we can now turn towards the main technical component of the proof of Proposition \ref{p:leading}.

\begin{proof} [Proof of Proposition \ref{p:leading}] Before proceeding with the proof, we fix some notations. Fix $s>0$, and set $n=\lfloor s \rfloor+1 \ge 1$ and  $\alpha=s-\lfloor s\rfloor\in [0,1)$ so that $s=n-1+\alpha$. Throughout the proof, we will denote $\Con$ to be positive constant depending only on $s,q$ -- we will not mention this further.  We will also use the big $O$ notation. For two complex-valued functions $f_1(t)$ and $f_2(t)$ and $\beta\in \R$, the equations $f_1(t)=(1+O(t^{\beta}))f_2(t)$ and $f_1(t)=f_2(t)+O(t^{\beta})$ have the following meaning: there exists a constant  $\Con>0$ such that for all large enough $t$,
	$$\left|\frac{f_1(t)}{f_2(t)}-1\right| \le \Con \cdot t^{\beta}, \mbox{ and } |f_1(t)-f_2(t)| \le \Con \cdot t^{\beta},$$ 
	respectively. The constant $\Con>0$  value may change from line to line.  
	
	For clarity we divide the proof into seven steps. In Steps 1 and 2, we provide the upper and lower bounds for $|\mathcal{A}_s(t)|$ and $\rl[\mathcal{A}_s(t)]$ respectively and complete the proof of \eqref{eq:lead}; in Steps 3--7, we verify the technical estimates assumed in the previous steps.
	
	\medskip
	
	\noindent\textbf{Step 1.} Recall $\calA_s(t)$ from \eqref{eq:cala}. The goal of this step is to provide a different expression for $\calA_{s}(t)$, which will be much more amenable to our analysis, as well as an upper bound for $|\calA_{s}(t)|$. By Proposition \ref{ppn:dkernel}, we have $\frac{\d^n}{\d\zeta^n}\tr(K_{\zeta,t})=\tr(K_{\zeta,t}^{(n)})$ and consequently using the expression in \eqref{def: kerdrv} we have
	\begin{align*}
		\calA_{s}(t):= (-1)^n\int_{1}^{e^{t\bq(\frac{s}2)}}  \frac{\zeta^{-\alpha}}{(2\pi\i)^2}\int_{\C(\tau^{1-\frac{\delta}{2}})}\int_{\delta-\i\infty}^{\delta +\i\infty}\Gamma(-u)\Gamma(1+u)(u)_n\zeta^{u-n}\frac{e^{tf(u,w)}\d u}{w-\tau^u w} \d w\d\zeta.
	\end{align*}
	where $\delta \in (0,1)$ is chosen to be less than $s$. We now proceed to deform the $u$-contour and $w$-contour sequentially.  As we explained in Remark \ref{r:pole}, the integrand has no poles when $u=1,2,\ldots, n-1$. Hence $u$-contour can be deformed to $(s - \i \infty, s+\i \infty)$  as $s=n-1+\alpha\in (0,n).$
	
	Next, for the $w$-contour, we wish to deform it from $\C(\tau^{1-\frac\delta2})$ to $\C(\tau^{1-\frac{s}2})$. In order to do so, we need to ensure that we do not cross any poles. We observe that the potential sources of poles lie in the exponent $f(u,w):=\frac{(q-p)}{1+w\tau^{-1}}-\frac{(q-p)}{1+\tau^{u-1}w}$ (recalled from \eqref{eq:contour_fn}) and in the denominator $w- \tau^u w.$ Since for any $w \in \C(\tau^{1-\frac{\delta'}{2}}),$ where $\delta'\in (\delta, s)$, and $ u \in (s -\i \infty, s +\i \infty),$ we have 
	$$|w - \tau^u w| \ge |w|- |\tau^uw|= \tau^{1-\frac{\delta'}{2}}(1-\tau^s) > 0, \quad  |1+w\tau^{-1}|\ge |w\tau^{-1}|-1=\tau^{-\frac{\delta'}2}-1>0, $$ 
	$$\mbox{ and } |1+\tau^{u-1}w|\ge 1- |\tau^{u-1}w| =1-\tau^{s-\frac{\delta'}2} > 0.$$
	Thus, we can deform the $w$-contour to $\C(\tau^{1-\frac{s}2})$ as well without crossing any poles. With the change of variable $u=s+\i y$, $w=\tau^{1-\frac{s}2}e^{\i\theta}$, and Euler's reflection formula we have	
	\begin{align}\label{eq:trip}
		\calA_{s}(t)= (-1)^n\int_{1}^{e^{t\bq(\frac{s}2)}} \frac{\zeta^{-1}}{4\pi^2}\int_{-\pi}^{\pi}\int_{\R}\frac{(s+\i y)_n\zeta^{\i y}}{\sin(-\pi (s+\i y))}e^{tf(s+\i y, \tau^{1-\frac{s}{2}}e^{\i \theta})}\frac{\d y}{1-\tau^{s+\i y} } \d \theta\d\zeta.
	\end{align}
	
	With this expression in hand, upper bound is immediate.   By Lemma \ref{l:kdbd} (\eqref{eq: idbd} specifically with $\rho\mapsto n-1$, $\delta\mapsto s$) pushing the absolute value inside the integrals  we see that
	\begin{align}\label{eq:up1}
		|\calA_{s}(t)|\le \Con\exp(-t\hq(s))\int_{1}^{e^{t\bq(\frac{s}2)}}  \frac{1}{\zeta}\d\zeta = \Con \cdot t\bq(\tfrac{s}2)\exp(-t\hq(s))
	\end{align}
	for some constant $\Con=\Con(q,s)>0$. Hence taking logarithm and dividing by $t$, we get 
	\begin{align}\label{eq:upper}
		\limsup\limits_{t\to\infty} |\mathcal{A}_s(t)| \le -\hq(s) = -(q-p)\frac{1-\tau^{\frac{s}2}}{1+\tau^{\frac{s}2}}.
	\end{align}
	
	\medskip
	
	\noindent\textbf{Step 2.} In this step, we provide a lower bound for $\rl[\calA_s(t)]$. Set $\e=t^{-2/5}>0$. For each $k\in \Z$, set $v_k=-\frac{2\pi}{\log\tau}k$ and consider the interval $V_k:=[v_k -\e^2,v_k+\e^2 ].$ Also set $A_{\e}:=\{\theta \in [-\pi,\pi] : |e^{\i\theta}-1|\le \e |\log\tau|\}$. We divide the triple integral in \eqref{eq:trip} into following parts
	\begin{equation} \label{eq:A}
		\calA_{s}(t) = \sum_{k \in \Z}(\mathbf{I})_k + (\mathbf{II}) + (\mathbf{III}),
	\end{equation}
	where 
	\begin{align}\label{eq:1k}
		(\mathbf{I})_k & : =  \int_{1}^{e^{t\bq(\frac{s}2)}} \int_{A_{\e}}\int_{V_k} \frac{(-1)^n}{4\pi^2\zeta}\frac{(s+\i y)_n\zeta^{\i y}}{\sin(-\pi (s+\i y))}\frac{e^{tf(s+\i y, \tau^{1-\frac{s}{2}}e^{\i \theta})}\d y}{1-\tau^{s+\i y} } \d \theta\d\zeta, \\ \label{eq:2k}
		(\mathbf{II}) & := \int_{1}^{e^{t\bq(\frac{s}2)}} \int_{A_{\e}}\int_{\R\setminus \cup_k V_k} \frac{(-1)^n}{4\pi^2\zeta}\frac{(s+\i y)_n\zeta^{\i y}}{\sin(-\pi (s+\i y))}\frac{e^{tf(s+\i y, \tau^{1-\frac{s}{2}}e^{\i \theta})}\d y}{1-\tau^{s+\i y} } \d \theta\d\zeta,\\ \label{eq:3k}
		(\mathbf{III}) & :=  \int_{1}^{e^{t\bq(\frac{s}2)}} \int_{[-\pi,\pi]\cap A_{\e}^c}\int_{\R} \frac{(-1)^n}{4\pi^2\zeta}\frac{(s+\i y)_n\zeta^{\i y}}{\sin(-\pi (s+\i y))}\frac{e^{tf(s+\i y, \tau^{1-\frac{s}{2}}e^{\i \theta})}\d y}{1-\tau^{s+\i y} } \d \theta\d\zeta.
	\end{align}
	In subsequent steps we obtain the following estimates for each integral. We claim that we have
	\begin{align}\label{eq:in1}
		(\mathbf{I})_0 = (1+O(t^{-\frac{1}{5}}))\frac{\Con_0}{\sqrt{t}}\exp(-t\hq(s)),
	\end{align}
	where $\hq(s)$ is defined in \eqref{eq:exp} and
	\begin{align}\label{eq:co1}
		\Con_0 := \sqrt{\frac{(1+\tau^{\frac{s}2})^3}{4\pi(q-p)(\tau^{\frac{3s}2-2}-\tau^{2s-2})}}\frac{(-1)^n(s)_n}{\sin(-\pi s)(1-\tau^s)} > 0.
	\end{align}
	When $s$ is an integer the above constant is defined in a limiting sense. Note that $\Con_0$ is indeed positive as $n=\lfloor s \rfloor +1$. Furthermore, we claim that we have the following upper bounds for the other integrals:
	\begin{align}\label{eq:in2}
		\sum_{k\in \Z\setminus \{0\}} |(\mathbf{I})_k| \le \Con t^{-\frac{13}{10}}\exp(-t\hq(s)).
	\end{align}
	where $v_k=-\frac{2\pi}{\log\tau}k$ and
	\begin{align}\label{eq:out}
		|(\mathbf{II})|, |(\mathbf{III})|\le \Con t\exp\left(-t\hq(s)\right)\exp(-\tfrac1\Con t^{\frac{1}5}).
	\end{align}
	Assuming the validity of \eqref{eq:in1}, \eqref{eq:in2} and \eqref{eq:out} we can complete the proof of lower bound for \eqref{eq:lead}. 
	Following the decomposition in \eqref{eq:A} we see that for all large enough $t$,
	\begin{align*}
		\rl [\calA_s(t)] & \ge \rl [(\mathbf{I})_0]- \sum_{k\in Z\setminus \{0\}} |(\mathbf{I})_k| - | (\mathbf{II})|-|(\mathbf{III})| \\ & \ge \tfrac{1}{\sqrt{t}}\exp(-th_q(s))\left[\tfrac1{2}\Con_0-\Con t^{-\frac{4}5}- \Con t^{\frac32}\exp(-\tfrac1{\Con}t^{\frac{3}5})\right] \ge \tfrac{\Con_0}{4\sqrt{t}}\exp(-t\hq(s)).
	\end{align*} 
	Taking logarithms and dividing by $t$ we get that $\liminf_{t\to \infty} \rl[\calA_s(t)] \ge -\hq(s)$. Combining with \eqref{eq:upper} we arrive at \eqref{eq:lead}.
	
	\medskip
	
	\noindent\textbf{Step 3.} In this step, we prove \eqref{eq:out}. Recall $(\mathbf{II})$ and $(\mathbf{III})$ defined in \eqref{eq:2k} and \eqref{eq:3k}. For each of them, we push the absolute value around each term of the integrand. We use \eqref{eq: idbd} from Lemma \ref{l:kdbd} to get
	\begin{align}\label{eq:ii}
		&|(\mathbf{II})| \le \Con \exp\bigg(t\sup_{\substack{y \in \R \setminus \cup_k V_k\\|e^{\i \theta }-1|\leq \e|\log\tau|}}\rl[f(s + \i y, \tau^{1-\frac{s}{2}}e^{\i \theta})]\bigg)\int_1^{e^{t\bq(\frac{s}2)}} \frac{\d\zeta}{\zeta},\\ &\label{eq:iii} |(\mathbf{III})|\le \Con \exp\bigg(t\sup_{\substack{y \in \R \\|e^{\i \theta }-1|> \e|\log\tau|}}\rl[f(s + \i y, \tau^{1-\frac{s}{2}}e^{\i \theta})]\bigg)\int_1^{e^{t\bq(\frac{s}2)}} \frac{\d\zeta}{\zeta}.
	\end{align}
	Note that in \eqref{eq:ii}, we have $|\tau^{\i y} -1|\ge  |\tau^{\i t^{-\frac{4}{5}}} -1 | \ge \frac12|\log\tau| t^{-\frac{4}{5}}$ for all large enough $t$. Meanwhile in \eqref{eq:iii},  $|\tau^{1- \frac{s}{2}}(e^{\i \theta} -1)|\ge \tau^{1- \frac{s}{2}}\e|\log \tau| = \tau^{1- \frac{s}{2}}|\log \tau|t^{-\frac{2}{5}}$. In either case, appealing to \eqref{max:ineq} in Lemma \ref{lem:max} with $\rho\mapsto s$ gives us that 
	\begin{align*}
		f(s, \tau^{1- \frac{s}{2}}) - \rl[f(s+\i y, \tau^{1-\frac{s}{2}}e^{\i \theta})] \ge \tfrac1\Con \cdot t^{-\frac{4}{5}}.
	\end{align*}
	Substituting $f(s, \tau^{1-\frac{s}{2}})$ with $-\hq(s)$ and evaluating the integrals in \eqref{eq:ii} and \eqref{eq:iii} gives us \eqref{eq:out}.
	
	\medskip
	
	\noindent\textbf{Step 4.} In this step and subsequent steps we prove \eqref{eq:in1} and \eqref{eq:in2}. Recall that $v_k=-\frac{2\pi}{\log \tau} k$ and $\e=t^{-\frac25}$. We focus on the $(\mathbf{I})_k$ integral defined in \eqref{eq:ik}. Our goal in this and next step is to show
	\begin{align}\label{eq:ik1}
		(\mathbf{I})_k =  (1+O(t^{-\frac15}))\frac{ \Con_0(k)}{2\pi\sqrt{t}}\int_{1}^{e^{t\bq(\frac{s}2)}} \frac{\zeta^{\i v_k}}{\zeta}\int_{-\e^2}^{\e^2}\zeta^{\i y}\exp(-t\hq(s+\i y)) \d y\d\zeta.
	\end{align}
	where 
	\begin{align}\label{eq:cok}
		\Con_0(k):=\sqrt{\frac{(1+\tau^{\frac{s}2})^3}{4\pi(q-p)(\tau^{\frac{3s}2-2}-\tau^{2s-2})}}\frac{(-1)^n(s+\i v_k)_n}{\sin(-\pi (s+\i v_k))(1-\tau^s)}
	\end{align}
	
	Towards this end, note that in the argument for \eqref{eq:up1}, we push the absolute value around each term of the integrand. Thus, the upper bound achieved in \eqref{eq:up1} guarantees that the triple integral in $(\mathbf{I})_k$ is absolutely convergent. Thereafter, Fubini's theorem allows us to switch the order of integration inside $(\mathbf{I})_k$. By a change-of-variables, we see that
	\begin{align*}
		(\mathbf{I})_k  =  (-1)^n\int_{1}^{e^{t\bq(\frac{s}2)}} \frac{\zeta^{\i v_k-1}}{4\pi^2}\int_{-\e^2}^{\e^2}\frac{(s+\i y+\i v_k)_n\zeta^{\i y}}{\sin(-\pi (s+\i y+\i v_k))}\int_{A_\e}\frac{e^{tf(s+\i y, \tau^{1-\frac{s}{2}}e^{\i \theta})}\d \theta}{1-\tau^{s+\i y} } \d y\d\zeta,
	\end{align*}
	where recall $A_{\e}=\{\theta \in [-\pi,\pi] : |e^{\i\theta}-1|\le \e |\log\tau|\}$. Note that in this case range of $y$ lies in a small window of $[-t^{-\frac{4}5},t^{-\frac45}]$. As $s$ is fixed, one can replace $(s+\i y+\i v_k)_n$, $\sin (-\pi(s+\i y+\i v_k))$, and $1-\tau^{s+\i y}$ by $(s+\i v_k)_n$, $\sin (-\pi(s+\i v_k))$, and $1-\tau^{s}$ with an expense of $O(t^{-\frac{4}5})$ term (which can be chosen independent of $k$). We thus obtain
	\begin{align}\label{eq:ik}
		(\mathbf{I})_k =  \frac{(-1)^n(s+\i v_k)_n(1+O(t^{-\frac45}))}{\sin(-\pi (s+\i v_k))(1-\tau^s)}\int_{1}^{e^{t\bq(\frac{s}2)}} \frac{\zeta^{\i v_k}}{4\pi^2\zeta}\int_{-\e^2}^{\e^2}\zeta^{\i y}\int_{A_\e}e^{tf(s+\i y, \tau^{1-\frac{s}{2}}e^{\i \theta})}\d \theta \d y\d\zeta.
	\end{align}	
	We now evaluate the $\theta$-integral in the above expression.  We claim that
	\begin{align}\label{theta1}
		\int_{A_\e} e^{tf(s+\i y, \tau^{1-\frac{s}{2}}e^{\i \theta})}\d \theta & = (1+O(t^{-\frac1{5}}))\sqrt{\frac{\pi(1+\tau^{\frac{s}2})^3}{t(q-p)(\tau^{\frac{3s}2-2}-\tau^{2s-2})}}\exp(-t\hq(s+\i y))
	\end{align}
	Note that \eqref{eq:ik1} follows from \eqref{theta1}. Hence we focus on proving \eqref{theta1} in next step.
	
	\medskip
	
	\noindent\textbf{Step 5.} In this step we prove \eqref{theta1}.	For simplicity we let $u=s+\i y$ temporarily. Taylor expanding the exponent appearing in l.h.s.~of \eqref{theta1} around $\theta=-\frac{y}{2}\log\tau$ and using the fact $\partial_{z}f(u,z)|_{z =\tau^{1-\frac{u}{2}}} = 0$, we get 
	\begin{align}
		\mbox{l.h.s.~of \eqref{theta1}}  & =
		\int_{A_\e}e^{tf(u, \tau^{1-\frac{u}{2}}e^{\i (\theta + \frac{y}{2}\log \tau)})}\d \theta \nonumber \\ & = \exp(tf(u, \tau^{1-\frac{u}2}))\int_{A_\e} \exp\left(-\frac t2\partial_z^2f(u,\tau^{1-\frac{u}2})(\theta+\tfrac{y}{2}\log\tau)^2+ O(t^{-\frac{1}{5}})\right)\d \theta.	\label{theta2}
	\end{align}
	Note that we have replaced the higher order terms by $O(t^{-\frac15})$ in the exponent above as $\theta, y$ are at most of the order $O(t^{-\frac25})$. Furthermore, for all $t$ large enough, 
	\begin{align*}
		A_\e & =\{\theta \in [-\pi,\pi] : |e^{\i\theta}-1|\le \e |\log\tau|\} \\ & = \{\theta \in [-\pi,\pi] : |\sin\tfrac{\theta}{2}|\le \tfrac12\e |\log\tau|\} \supset \{\theta \in [-\pi,\pi] : |\theta|\le \e |\log\tau|\}
	\end{align*} 
	As $y\in [-\e^2,\e^2]$, we see that $A_\e \supset \{\theta \in [-\pi,\pi] : |\theta+\frac{y}{2}\log\tau|\le \frac12\e|\log\tau|\}$ for all large enough $t$. Thus on $A_\e^c$ we have $|\theta+\frac{y}2\log\tau| \ge \frac12t^{-\frac25}|\log\tau|$.
	Furthermore for small enough $y$, by \eqref{eq;deri}, we have $\rl[\partial_z^2 f(u,\tau^{1-\frac{u}2})]>0$. Hence the above integral can be approximated by Gaussian integral. In particular, we have
	\begin{align}
		\mbox{r.h.s.~of \eqref{theta2}} & = (1 +O(t^{-\frac15}))\exp(tf(u, \tau^{1-\frac{u}2}))\sqrt{\frac{2\pi}{t\partial_z^2f(u,\tau^{1-\frac{u}2})}} \label{e1}
	\end{align}
	Observe that as $u=s+\i y$ and $y$ is at most $O(t^{-\frac45})$, $\partial_z^2f(u,\tau^{1-\frac{u}2})$ in r.h.s.~of \eqref{e1} can be replaced by $\partial_z^2f(s,\tau^{1-\frac{s}2})$  by adjusting the order term. Recall the expression for $\partial_z^2f(s,\tau^{1-\frac{s}2})$ from \eqref{eq;deri} and observe that from the definition of $f$ and $\hq$ from \eqref{eq:contour_fn} and \eqref{eq:exp} we have $f(u,\tau^{1-\frac{u}2})=\hq(s+\i y)$. We thus arrive at \eqref{theta1}. 
	
	\medskip
	
	\noindent\textbf{Step 6.}  In this step and we prove \eqref{eq:in1} and \eqref{eq:in2} starting from the expression of $(\mathbf{I})_k$ obtained in \eqref{eq:ik1}. As $y$ varies in the window of $y\in [t^{-\frac45},t^{-\frac45}]$, by Taylor expansion we may replace $t\hq(s+\i y)$ appearing in the r.h.s.~of \eqref{eq:ik1} by  $t(\hq(s)+\i y\hq'(s))$ at the expense of an $O(t^{-\frac35})$ term. Upon making a change of variable $r=\log\zeta-t\hq'(s)$ we thus have  
	\begin{align}
		(\mathbf{I})_k & =  (1+O(t^{-\frac15}))\frac{ \Con_0(k)}{2\pi\sqrt{t}}\exp(-t\hq(s))\int_{-t\hq'(s)}^{{t\bq(\frac{s}2)}-t\hq'(s)} 	e^{\i v_k (r+t\hq'(s))}\int_{-\e^2}^{\e^2}e^{\i y r} \d y\d r \nonumber \\ & \label{eq:ik11} =  (1+O(t^{-\frac15}))\frac{ \Con_0(k)}{2\pi\sqrt{t}}\exp(-t\hq(s))\int_{-t\hq'(s)}^{{t\bq(\frac{s}2)}-t\hq'(s)} 	e^{\i v_k (r+t\hq'(s))}\frac{e^{\i \e^2 r}-e^{-\i \e^2 r}}{\i r}\d r. 
	\end{align}
	We claim that for $k=0$, (which implies $v_k=0$) we have
	
	\begin{align}\label{eq:c1}
		\int_{-t\hq'(s)}^{{t\bq(\frac{s}2)}-t\hq'(s)} \frac{e^{\i \e^2 r}-e^{-\i \e^2 r}}{\i r}\d r = 2\pi(1+O(t^{-\frac15}))
	\end{align}
	For $k\neq 0$, we have
	\begin{align}\label{eq:c2}
		\left|\int_{-t\hq'(s)}^{{t\bq(\frac{s}2)}-t\hq'(s)} 	e^{\i v_k (r+t\hq'(s))}\frac{e^{\i \e^2 r}-e^{-\i \e^2 r}}{\i r}\d r\right| \le \Con t^{-\frac{4}5}
	\end{align}
	where $\Con>0$ can be chosen free of $k$. Assuming \eqref{eq:c1} and \eqref{eq:c2} we may now complete the proof of \eqref{eq:in1} and \eqref{eq:in2}. Indeed, for $k=0$ upon observing that $\Con_0=\Con_0(0)$ (recall \eqref{eq:co1} and \eqref{eq:cok}), in view of \eqref{eq:ik11} and \eqref{eq:c1} we get \eqref{eq:in1}. Whearas for $k\neq 0$, thanks to the estimate in \eqref{eq:c2}, in view of \eqref{eq:ik11}, we have
	\begin{align}\label{eq:ksum}
		\sum_{k\in \Z\setminus \{0\}} |(\mathbf{I})_k| \le \Con t^{-\frac{13}{10}} \exp(-t\hq(s))\sum_{k\in \Z\setminus \{0\}} |\Con_0(k)|.
	\end{align}
	For $y\neq 0$, $|\frac{(s+\i y)_n}{\sin(-\pi(s+\i y))}|\le \Con |y|^ne^{-|y|/\Con}$ forces r.h.s.~of \eqref{eq:ksum} to be summable proving \eqref{eq:in2}. 
	
	\medskip
	
	\noindent\textbf{Step 7.} In this step we prove \eqref{eq:c1} and \eqref{eq:c2}. Recalling that $\e^2=t^{-\frac45}$, we see that
	
	\begin{align}\label{eq:k0}
		\int_{-t\hq'(s)}^{{t\bq(\frac{s}2)}-t\hq'(s)} \frac{e^{\i \e^2 r}-e^{-\i \e^2 r}}{\i r}\d r = \int_{-t^{1/5}\hq'(s)}^{{t^{1/5}\bq(\frac{s}2)}-t^{1/5}\hq'(s)} 	\frac{2\sin r}{r}\d r.
	\end{align}
	Following the definition of $\hq$ and $\bq$ in Proposition \ref{p:htau} we observe that that $-\hq'(s)=\frac{\tau^{\frac{s}2}\log\tau}{(1+\tau^{\frac{s}2})^2}<0$ and 
	$$B_q(s)-\hq'(s)=\frac{1-\tau^s+\tau^{\frac{s}{2}}s\log \tau}{s(1+\tau^{\frac{s}2})}=-s\bq'(s)>0,$$ 
	where $\bq'(s)<0$ follows from \eqref{eq:bq}.
	Thus as $\bq$ is strictly decreasing (Proposition \ref{p:htau} \ref{a}) we have $\bq(\frac{s}{2}) > \bq(s) > \hq'(s)$. Thus the integral on r.h.s.~of \eqref{eq:k0} can be approximated by $(1+O(t^{-1/5}))\int_{\R} \frac{2\sin r}{r}\d r= 2\pi(1+O(t^{-1/5}))$. This proves \eqref{eq:c1}. We now focus on proving \eqref{eq:c2}. Towards this end, we divide the integral appearing in \eqref{eq:c2} into three regions as follows
	\begin{equation}
		\label{eq:ik15}	
		\begin{aligned}
			\mbox{l.h.s.~of \eqref{eq:c2}} & \le \left|\int_{-t\hq'(s)}^{-1} 	e^{\i v_k (r+t\hq'(s))}\frac{e^{\i \e^2 r}-e^{-\i \e^2 r}}{\i r}\d r\right|  +\left|\int_{-1}^{1} 	e^{\i v_k (r+t\hq'(s))}\frac{e^{\i \e^2 r}-e^{-\i \e^2 r}}{\i r}\d r\right|\\ & \hspace{3cm}+\left|\int_{1}^{{t\bq(\frac{s}2)}-t\hq'(s)} 	e^{\i v_k (r+t\hq'(s))}\frac{e^{\i \e^2 r}-e^{-\i \e^2 r}}{\i r}\d r\right|. 
		\end{aligned}
	\end{equation} Note that for the second term appearing in r.h.s.~of \eqref{eq:ik15} can be bounded by $4 t^{-\frac45}$ using
	$$\left|\int_{-1}^{1} e^{\i v_k (r+t\hq'(s))}\frac{2\sin(\e^2r)}{r}\d r\right| \le \int_{-1}^{1} \left|\frac{2\sin(\e^2r)}{r}\right|\d r \le 4\e^2=4t^{-\frac{4}5}.$$
	For the first term  appearing in r.h.s.~of \eqref{eq:ik15}, by making a change of variable $r \mapsto r\frac{v_k-\e^2}{v_k+\e^2}$ we observe the following identity.
	\begin{align*}
		\int_{-t\hq'(s)}^{-1} 	e^{\i v_k (r+t\hq'(s))}\frac{e^{\i \e^2 r}}{\i r}\d r = \int_{-t\hq'(s)\frac{v_k+\e^2}{v_k-\e^2}}^{-\frac{v_k+\e^2}{v_k-\e^2}} 	e^{\i v_k (r+t\hq'(s))}\frac{e^{-\i \e^2 r}}{\i r}\d r
	\end{align*}
	This leads to 
	\begin{align*}
		\int_{-t\hq'(s)}^{-1} 	e^{\i v_k (r+t\hq'(s))}\frac{e^{\i \e^2 r}-e^{-\i\e^2 r}}{\i r}\d r & = \int_{-t\hq'(s)}^{-t\hq'(s)\frac{v_k+\e^2}{v_k-\e^2}} e^{\i v_k (r+t\hq'(s))}\frac{e^{-\i \e^2 r}}{\i r}\d r \\ & \hspace{2cm}+ \int_{-\frac{v_k+\e^2}{v_k-\e^2}}^{-1} 	e^{\i v_k (r+t\hq'(s))}\frac{e^{-\i \e^2 r}}{\i r}\d r
	\end{align*}
	In the first integral the length of the interval is $O(t^{1/5})$. However, the integrand itself is $O(t^{-1})$. For the second integral, the length of the interval is $O(t^{-4/5})$, and the integrand itself is $O(1)$. Note that this is only possible when $k\neq 0$ (forcing $v_k\neq 0$). And indeed all the $O$ terms can be taken to be free of $v_k$ (and hence of $k$). Combining this we get that the first term appearing in r.h.s~of \eqref{eq:ik15} can be bounded by $\Con t^{-\frac45}$. An exact analogous argument provides the same bound for the third term in r.h.s.~of \eqref{eq:ik15} as well. This proves \eqref{eq:c2} completing the proof. 
\end{proof}

	\section{Bounds for the Higher order terms}	\label{sec:higher}
	The goal of this section is to establish bounds for the higher-order term $\mathcal{B}_s(t)$ defined in (\ref{eq:Calb}). First, recall the Fredholm determinant formula from \eqref{eq:f-series}. Using the $\tr(K_{\zeta,t}^{\wedge L})$ notation from \eqref{eq:fdhm} we may rewrite $\calB_s(t)$ as follows.
	\begin{align}\label{eq: bst}
	\mathcal{B}_s(t) = (-1)^n\int_1^{e^{t\bq(\frac{s}{2})}}\zeta^{-\alpha}\frac{\d^n}{\d \zeta^n}\bigg[1 + \sum_{L =2}^{\infty}\tr(K^{\wedge L}_{\zeta. t})\bigg]\d \zeta.
	\end{align}
	We claim that we could exchange the various integrals, derivatives and sums appearring in the r.h.s. of  \eqref{eq: bst} and obtain $\mathcal{B}_s(t)$ through term-by-term differentiation, i.e. 
\begin{align}\label{eq:bfinal}
	\mathcal{B}_s(t) = (-1)^n \sum_{L=2}^{\infty}\int_1^{e^{t\bq(\frac{s}{2})}}\zeta^{-\alpha}\partial_{\zeta}^n(\tr(K_{\zeta,t}^{\wedge L}))\d \zeta.
\end{align}
	 Towards this end, we devote Section \ref{intrchnge} to its  justification. Following the technical lemmas in Section \ref{intrchnge}, we proceed to prove Proposition \ref{p:ho} in Section \ref{pf: p:ho}.
\subsection{Interchanging  sums, integrals and derivatives}\label{intrchnge}
Recall from (\ref{def: kerdrv}) the definition of $K_{\zeta, t}^{(n)}.$  As a starting point of our analysis, we introduce the following notations before providing the bounds on $|\partial_{\zeta}^n \tr(K_{\zeta, t}^{\wedge L})|.$
For any $n,L\in \Z_{>0}$, define
\begin{equation} \label{mln}
	\mathfrak{M}(L, n) := \{\vec{m} = (m_1, \ldots, m_L) \in (\Z_{\geq 0})^L: m_1 + \cdots + m_L = n\}
\mbox{ and }
	\binom{n}{\vec{m}}: = \frac{n!}{m_1!\cdots m_L!}.
\end{equation}
Furthermore, for any $L \in \Z_{>0},$ $\zeta \in \R_{>0}$ and $\vec{m} \in \mathfrak{M}(L, n)$, let
\begin{equation}\label{def: Im}
I_{\zeta}(\vec{m}) := {\int\ldots\int}\det(K_{\zeta, t}^{(m_i)}(w_i, w_j))_{i, j =1}^L\prod_{i = 1}^L\d w_i 
\end{equation}
where $w_i$-contour lies on $\C(\tau^{1-\frac{\delta}2})$. We also set $|\vec{m}|_{>0}: = |\{i \mid i \in \Z \cap [1, L], m_i >0\}|,$ i.e. the number of positive $m_i$ in $\vec{m}.$

To begin with, the next two lemma investigate the term-by-term $n$-th derivatives of  $\tr(K^{\wedge L}_{\zeta,t})$ that appear on the r.h.s. of \eqref{eq:bfinal}. The following should be regarded as a higher order version of Proposition \ref{ppn:dkernel}. 
\begin{proposition}\label{p:trL}
	Fix $n,L \in \Z_{>0}$ and let $\mathfrak{M}(L, n)$ be defined as in $(\ref{mln}).$ Recall the function $\bq(x)$ from Proposition \ref{p:htau}. For any $t > 0$, the function $\zeta\mapsto \tr(K_{\zeta, t}^{\wedge L})$ is infinitely differentiable at  each $\zeta \in [1, e^{t\bq(\frac{s}{2})}]$, with
	\begin{equation}\label{eq: exdi}
		\partial_\zeta^n \tr(K_{\zeta, t}^{\wedge L}) = \frac{1}{L!}\sum_{\vec{m}\in \mathfrak{M}(L, n)}\binom{n}{\vec{m}} I_{\zeta}(\vec{m}),
	\end{equation}
 where the r.h.s of \eqref{eq: exdi} converges absolutely uniformly. Furthermore, there exists a constant $\Con = \Con(n, \delta, q)>0$ such that for all $\vec{m}\in \mathfrak{M}(L,n)$ we have
 \begin{equation}\label{eq:exdi2}
 	|I_\zeta(\vec{m})| \le \Con^L L^{\frac{L}2}\zeta^{L\delta-n}\exp(-t\hq(\delta)), \quad 
|\partial_\zeta^n \tr(K_{\zeta, t}^{\wedge L})| \leq \frac{1}{L!}\Con^L L^n L^{\frac{L}{2}}\zeta^{L\delta-n}\exp(-tL\hq(\delta)).
 \end{equation} 
\end{proposition}
\begin{proof} The proof idea is same as that of Proposition \ref{ppn:dkernel}, but it's more cumbersome notationally. For clarity we split the proof into four steps. In the first step, we introduce some necessary notations. In Steps 2-3, we prove \eqref{eq: exdi} and in the final step, we prove \eqref{eq:exdi2}.
	
	\medskip

	\noindent\textbf{Step 1.} In this step we summarize the notation we will require in the proof of \eqref{eq: exdi}. We fix $L\in \Z_{>0}, \delta\in (0,1),t>0$, and $\zeta_1, \zeta_2 > 0$ and recall $\bq(x)$ from Proposition \ref{p:htau}. 
	
	We define $\vec\xi_k \in [1,e^{tB_q(\frac{s}{2})}]^L$ to be the vector whose first $k$ entries are $\zeta_2$ and the rest $L-k$ entries are $\zeta_1$:
	$$\vec\xi_k := (\xi_{k,1},\xi_{k,2},\ldots,\xi_{k,L}):=(\ \underbrace{ \zeta_2\ , \ \zeta_2\ ,\ \ldots\ ,\ \zeta_2}_{k \mbox{ times}}\ ,\ \underbrace{\zeta_1 \ ,\ \zeta_1\ ,\ \ldots \ ,\ \zeta_1 }_{L-k \mbox{ times}}\ ), \quad  k=0,1,\ldots,L.$$
	For any $\vec{m}=(m_1,m_2,\ldots,m_L)\in (\Z_{\ge 0})^L$ we define the following integral of mixed parameters
	\begin{equation}\label{def: mIm}
		I_{\zeta_1,\zeta_2}^{(k)}(\vec{m}) := {\int\ldots\int}\det(K_{\xi_{k,i}, t}^{(m_i)}(w_i, w_j))_{i, j =1}^L\prod_{i = 1}^L\d w_i. 
	\end{equation}
where $w_i$-contour lies on $\C(\tau^{1-\frac{\delta}2})$.  $I_{\zeta_1,\zeta_2}^{(k)}(\vec{m})$ serves as an interpolation between $I_{\zeta_1}(\vec{m})$ and $I_{\zeta_2}(\vec{m})$ defined in \eqref{def: Im} as $k$ increases from 0 to $L$ where the parameters $\zeta$ are now allowed to be different for different rows in the determinant.
	
	\medskip
	
	 We next define $\vec{e}_k=(e_{k,1},e_{k,2},\ldots,e_{k,L})$ to be the unit vector with $1$ in the $k$-th position and $0$ elsewhere. With the above notations in place, for each $j,k \in \{1,2,\ldots,L\}$ and $\vec{m}\in (\Z_{\ge 0})^L$ we set 
	 \begin{align}\label{eq:d4}
	 	&\mathfrak{L}_{\zeta_1,\zeta_2}^{(1)}(\vec{m};k)  :=\frac1{\zeta_2-\zeta_1}\left[I_{\zeta_1,\zeta_2}^{(k)}(\vec{m})-I_{\zeta_1,\zeta_2}^{(k-1)}(\vec{m})-(\zeta_2-\zeta_1)I_{\zeta_1,\zeta_2}^{(k-1)}(\vec{m}+\vec{e}_k)\right],\\  \label{eq:d42}
	 	&\mathfrak{L}_{\zeta_1,\zeta_2}^{(2)}(\vec{m};j,k)  := I_{\zeta_1,\zeta_2}^{(j)}(\vec{m}+\vec{e}_k)-I_{\zeta_1,\zeta_2}^{(j-1)}(\vec{m}+\vec{e}_k).
	 \end{align}
	Note that we define \eqref{eq:d4} modelling after $D_{\zeta_1,\zeta_2}$ in the proof of Proposition \ref{ppn:dkernel}. Here, the only differences between the three determinants of  the respective $I_{\zeta_1,\zeta_2}^{(k)}(\vec{m})$'s lie in the $k$-th row, i.e. $K_{\zeta_2,t}^{(m_k)}$ v.s. $K_{\zeta_1,t}^{(m_k)}$ v.s. $K_{\zeta_1,t}^{(m_k+1)}.$ So we have isolated the differences and tried to reduce the question of differentiability to row-wise in \eqref{eq:d4}. Meanwhile, \eqref{eq:d42} ``measures" the distance between $I_{\zeta_1,\zeta_2}^{(k)}(\vec{m}+\vec{e}_k)$ and $I^{(k-1)}_{\zeta_1,\zeta_2}(\vec{m}+\vec{e}_k)$ where they differ only in $\xi_{k,k} = \zeta_2$ or $\zeta_1$ for $K_{\xi_{k,k},t}^{(m_k)}$ on the $k$-th row of the determinant. 
	
	We finally remark that all the $w_i$-contours in the integrals appearing throughout the proof are on $\C(\tau^{1-\frac\delta2})$ -- we will not mention this further. We would also drop $(w_i,w_j)$ from $K_{\bullet,t}^{(m_i)}(w_i,w_j)$ when it is clear from the context.
	
	\medskip

	\noindent\textbf{Step 2.} We show the infinite differentiability of $\tr(K_{\zeta, t}^{\wedge L})$ by proving \eqref{eq: exdi} in this step.  The proof proceeds via induction on $n$. When $n=0$, observe that \eqref{eq: exdi} recovers the formula of $\tr(K_{\zeta, t}^{\wedge L}).$ This constitutes the base case. To prove the induction step, suppose \eqref{eq: exdi} holds for $n = N$. Then for $n = N+1$, we fix $\zeta_1, \zeta_2 > 0$. Without loss of generality, we assume $\zeta_1+1 > \zeta_2 > \zeta_1$ and consider
 \begin{align}\label{eq:d0}
 	D_{\zeta_1, \zeta_2} & := \frac{\partial_{\zeta}^N \tr(K_{\zeta_2, t}^{\wedge L})- \partial_{\zeta}^N \tr(K_{\zeta_1,t}^{\wedge L}) }{\zeta_2 - \zeta_1} - \frac{1}{L!}\sum_{\vec{m}\in \mathfrak{M}(L, N+1)}\binom{N+1}{\vec{m}}I_{\zeta_1}(\vec{m}).
 \end{align}
To prove \eqref{eq: exdi}, it suffices to show $|D_{\zeta_1,\zeta_2}| \to 0$ as $\zeta_2 \to \zeta_1$. Towards this end, we first claim that for all $\vec{m}\in \mathfrak{M}(L,N)$ and for all $j,k \in \{1,2,\ldots,L\}$ we have
\begin{align}\label{eq:last}
	\big|\mathfrak{L}_{\zeta_1,\zeta_2}^{(1)}(\vec{m};k)\big| \to 0, \mbox{ and } \big|\mathfrak{L}_{\zeta_1,\zeta_2}^{(2)}(\vec{m};j,k)\big| \to 0,\mbox{ as }\zeta_2\to \zeta_1,
\end{align}
where $\mathfrak{L}_{\zeta_1,\zeta_2}^{(1)}(\vec{m};k)$ and $\mathfrak{L}_{\zeta_1,\zeta_2}^{(2)}(\vec{m};j,k)$ are defined in \eqref{eq:d4} and \eqref{eq:d42} respectively. We postpone the proof of \eqref{eq:last} to the next step. Assuming its validity, we now proceed to complete the induction step. 

Towards this end, we first manipulate the expression appearing in r.h.s.~of \eqref{eq:d0}. A simple combinatorial fact shows
 $$\sum_{\vec{m}\in \mathfrak{M}(L, N+1)}\binom{N+1}{\vec{m}} I_{\zeta_1}(\vec{m}) = \sum_{k=1}^L\sum_{\vec{m}\in \mathfrak{M}(L, N)}\binom{N}{\vec{m}}I_{\zeta_1}(\vec{m} + \vec{e}_k),$$  
 where $\vec{e}_k$ is defined in Step 1. Substituting this combinatorics back into the r.h.s. of \eqref{eq:d0} and using the induction step for $n=N$, allows us to rewrite $D_{\zeta_1, \zeta_2}$ as follows:
 \begin{align}\label{eq:d2}
 	\mbox{r.h.s. of }\eqref{eq:d0} = \frac{1}{L!}\sum_{\vec{m}\in \mathfrak{M}(L, N)}\binom{N}{\vec{m}}\left[\frac{I_{\zeta_2}(\vec{m}) - I_{\zeta_1}(\vec{m})}{\zeta_2 - \zeta_1} - \sum_{k=1}^L I_{\zeta_1}(\vec{m}+\vec{e}_k) \right].
 \end{align}
 Recalling the definition of $I_{\zeta}(\vec{m})$ in \eqref{def: Im} and that of  $I_{\zeta_1,\zeta_2}^{(k)}(\vec{m})$ in \eqref{def: mIm}, we see that $\sum_{k=1}^L [I_{\zeta_1,\zeta_2}^{(k)}(\vec{m})-I_{\zeta_1,\zeta_2}^{(k-1)}(\vec{m})]$ telescopes to $I_{\zeta_2}(\vec{m}) - I_{\zeta_1}(\vec{m})$. Furthermore, if we recall  $\mathfrak{L}_{\zeta_1,\zeta_2}^{(1)}(\vec{m};k)$ and $\mathfrak{L}_{\zeta_1,\zeta_2}^{(2)}(\vec{m};j,k)$ from \eqref{eq:d4} and \eqref{eq:d42} respectively, we observe that $$I_{\zeta_1,\zeta_2}^{(k-1)}(\vec{m}+\vec{e}_k)-I_{\zeta_1}(\vec{m}+\vec{e}_k)=I_{\zeta_1,\zeta_2}^{(k-1)}(\vec{m}+\vec{e}_k)-I_{\zeta_1,\zeta_2}^{(0)}(\vec{m}+\vec{e}_k)=\sum_{j=1}^k \mathfrak{L}_{\zeta_1,\zeta_2}^{(2)}(\vec{m};j,k).$$ Combining these observations, we have
 \begin{align}
 \mbox{r.h.s.~of \eqref{eq:d2}} & =\frac{1}{L!}\sum_{\vec{m}\in \mathfrak{M}(L, N)}\binom{N}{\vec{m}}\sum_{k=1}^L \frac{\left[I_{\zeta_1,\zeta_2}^{(k)}(\vec{m})-I_{\zeta_1,\zeta_2}^{(k-1)}(\vec{m})-(\zeta_2-\zeta_1)I_{\zeta_1}(\vec{m}+\vec{e}_k)\right]}{\zeta_2-\zeta_1}\nonumber \\ & = \frac{1}{L!}\sum_{\vec{m}\in \mathfrak{M}(L, N)}\binom{N}{\vec{m}}\sum_{k=1}^L \left[\mathfrak{L}^{(1)}_{\zeta_1,\zeta_2}(\vec{m};k)+\sum_{j=1}^{k-1} \mathfrak{L}_{\zeta_1,\zeta_2}^{(2)}(\vec{m};j,k)\right]. \label{eq:d3}
 \end{align} 

Clearly r.h.s.~of \eqref{eq:d3} goes to zero as $\zeta_2\to \zeta_1$ whenever \eqref{eq:last} is true. Thus by induction we have \eqref{eq: exdi}.

\medskip

\noindent\textbf{Step 3.} In this step we prove \eqref{eq:last}. Recall $\mathfrak{L}_{\zeta_1,\zeta_2}^{(1)}(\vec{m};k)$ from \eqref{eq:d4}. Following the definition of $I_{\zeta_1,\zeta_2}^{(k)}(\vec{m})$ from \eqref{def: mIm} we have
\begin{equation*}
	\begin{aligned}
	\big|\mathfrak{L}_{\zeta_1,\zeta_2}^{(1)}(\vec{m};k)\big| & \le \int\cdots\int \frac1{\zeta_2-\zeta_1}\left|\det(K_{\xi_{k,i}, t}^{(m_i)})_{i, j =1}^L-\det(K_{\xi_{k-1,i}, t}^{(m_i)})_{i, j =1}^L\right. \\ & \hspace{4cm}\left.-(\zeta_2-\zeta_1)\det(K_{\xi_{k-1,i}, t}^{(m_i+e_{k,i})})_{i, j =1}^L\right|\prod_{i=1}^L \d w_i.
\end{aligned}
\end{equation*}
Recall that in the above expression, up to a constant, the three determinants differ only in the $k$-th row. Hence the above expression can be written as $\int\cdots\int |\det(A)|\prod_{i=1}^L \d w_i$, where the entries of $A$ are given as follows: 
\begin{align*}
	A_{i,j} & = K_{\zeta_2,t}^{(m_i)}(w_i,w_j), \quad i<k, \quad A_{i,j}  = K_{\zeta_1,t}^{(m_i)}(w_i,w_j), \quad i>k, \\
	A_{k,j}  & =\frac1{\zeta_2-\zeta_1}[K_{\zeta_2,t}^{(m_k)}(w_k,w_j)-K_{\zeta_1,t}^{(m_k)}(w_k,w_j)-(\zeta_2-\zeta_1)K_{\zeta_1,t}^{(m_k+1)}(w_k,w_j)] \\ & =\frac1{2\pi\i}\int_{\delta-\i\infty}^{\delta+\i\infty}\Gamma(-u)\Gamma(1+u)R_{\zeta_1,\zeta_2;m_k}(u)e^{tf(u,w_k)}\frac{\d u}{w_j-\tau^u w_k}
\end{align*}
where $R_{\zeta_1,\zeta_2;m_k}(u)$ is same as in \eqref{eq:rdef}. As $m_i$'s are at most $n$, by Lemma \ref{l:kdbd} (\eqref{eq: kdbd} specifically), we can get a constant $\Con>0$ depending only on $n,\delta,$ and $q$, so that
 $$|A_{i,j}| \le \Con(\zeta_1^{\delta-m_k}+\zeta_2^{\delta-m_k})\exp(-t\hq(\delta))\le \Con(1+\zeta_2^{\delta})\exp(-t\hq(\delta))$$ for all $i\neq k$. For $A_{k,j}$, we follow the same argument as in Proposition \ref{ppn:dkernel} (along the lines of \eqref{eq: d}) to get
	\begin{align*}
	|A_{k,j}|  & \le \frac{\tau^{1-\frac{\delta}{2}}}{2\pi}\int_{\zeta_1}^{\zeta_2}\left|\sigma^{\delta+\i y-m_k-2}\right|\d\sigma \cdot \max_{w_j,w_k\in \C(\tau^{1-\frac\delta2})}\int_{\R} \left|\frac{(\delta+\i y)_{m_k+2}}{\sin(-\pi(\delta+\i y))}  e^{tf(\delta+\i y,w_k)}\right|\frac{\d y}{|w_j-\tau^{\delta+\i y} w_k|}.
\end{align*}
Note that by Lemma \ref{l:kdbd} (\eqref{eq: idbd} specifically) we see that the above maximum is bounded by $\Con \exp(-t\hq(\delta))$ where again as $m_i$'s are at most $n$, the constant $\Con$ can be chosen dependent only on $n$,$\delta,$ and $q$.   Since $|\sigma^{u-n-2}|= |\sigma^{\delta-m_k-2}|\le |\zeta_1^{\delta-m_k-2}| \le |\zeta_1^{\delta-2}|$ over the interval $[\zeta_1, \zeta_2]$ for $\delta \in (0,1)$, we obtain $$|A_{k,j}| \le \Con \exp(-t\hq(\delta))\int_{\zeta_1}^{\zeta_2} |\sigma|^{\delta-m_k-2}\d\sigma \le \Con\exp(-t\hq(\delta))\zeta_1^{\delta-2}(\zeta_2-\zeta_1).$$
As all the above estimates on $|A_{i,j}|$ are uniform in $w_i$'s, using Hadamard inequality we have
\begin{align*}
	\int\cdots\int |\det(A)|\prod_{i=1}^L \d w_i & \le \Con^L L^{\frac{L}2}\exp(-tL\hq(\delta))(1+\zeta_2^{\delta})^{L-1}\zeta_1^{\delta-2}(\zeta_2-\zeta_1)
\end{align*}
Taking $\zeta_2\to \zeta_1$ above, we get the first part of \eqref{eq:last}. The proof of the second part of \eqref{eq:last} follows similarly by observing that the corresponding determinants also differ only in one row. One can then deduce the second part of \eqref{eq:last} using the uniform estimates of the kernel and difference of kernels given in \eqref{eq: kdbd} and \eqref{eq: kcz} respectively. As the proof follows exactly in the lines of above arguments, we omit the technical details. 

\medskip

\noindent\textbf{Step 4.} In this step we prove \eqref{eq:exdi2}. 

Recall the definition of $I_{\zeta}(\vec{m})$ from \eqref{def: Im}. By Hadamard's inequality and Lemma \ref{l:kdbd} we have
 \begin{equation}\label{eq:hd}
 	\begin{aligned}
 |\det(K_{\zeta, t}^{(m_i)})_{i, j = 1}^L |& \le L^{\frac{L}{2}} \prod_{i=1}^L\max_{w_i,w_j\in \C(\tau^{1-\frac\delta2})} |K_{\zeta, t}^{(m_i)}(w_i, w_j)| \\ & \le  L^{\frac{L}{2}}\prod_{i=1}^L\Con \zeta^{\delta-m_i}\exp(- t \hq(\delta))=\Con^L L^{\frac{L}{2}}\zeta^{L\delta-n}\exp(- t \hq(\delta)),
 \end{aligned}
 \end{equation}
 where the last equality follows as $\sum_{i=1}^L m_i=n$. Note that here also $\Con>0$ can be chosen to be dependent only on $n$, $\delta$, and $q$ as $m_i$'s are at most $n$. Recall that $w_i$-contour in $I_{\zeta}(\vec{m})$ lies on $\C(\tau^{1-\frac\delta2})$. Thus in view of \eqref{eq:hd} adjusting the constant $\Con$ we obtain first inequality of \eqref{eq:exdi2}. 
 
 For the second inequality, We observe the following recurrence relation: 
 \begin{equation}\label{eq:rec}
 	|\mathfrak{M}(L, n)| = |\{\vec{m} = (m_1, \ldots, m_L)\in \Z_{\geq 0}^L, \sum_{i = 1}^L m_i= n\}| \leq L\cdot|\mathfrak{M}(L, n-1)|.
 \end{equation}
 It follows immediately that $|\mathfrak{M}(L, n)| \leq L^n.$ Observe that for each $\vec{m}\in \mathfrak{M}(L,n),$ $\binom{n}{\vec{m}}$ is bounded from above by $n!$. Thus collectively with \eqref{eq: exdi} we have 
 \begin{align*}
 	|\partial_{\zeta}^n\tr(K_{\zeta,t}^{\wedge L})| \le \frac{n!L^n}{L!}\max_{\vec{m}\in \mathfrak{M}(L,n)}|I_{\zeta}(\vec{m})|.
 \end{align*}
Applying the first inequality of \eqref{eq:exdi2} above leads to the second inequality of \eqref{eq:exdi2} completing the proof.

\end{proof}
 \begin{lemma}\label{p: d-s}
 	Fix $n\in \Z_{> 0}$, $\zeta \in [1, e^{t\bq(\frac{s}{2})}],$ and $t > 0$. Then
 	\begin{equation*}
 		\partial_{\zeta}^n \bigg(\sum_{L = 1}^{\infty}\tr(K_{\zeta,t}^{\wedge L}) \bigg) = \sum_{L =1}^{\infty}\partial_{\zeta}^n(\tr(K_{\zeta,t}^{\wedge L}) ).
 	\end{equation*}
 \end{lemma}
\begin{proof}
	On account of \cite[Proposition 4.2]{dt19}), it suffices to verify the following conditions:
	\begin{enumerate}
		\item $\sum_{L = 1}^{\infty}\tr(K_{\zeta,t}^{\wedge L})$ converges absolutely pointwise for $\zeta \in [1, e^{t\bq(\frac{s}{2})}];$
		\item  the absolute derivative series $ \sum_{L =1}^{\infty}\partial_{\zeta}^n(\tr(K_{\zeta,t}^{\wedge L}) )$ converges uniformly for $\zeta \in [1, e^{t\bq(\frac{s}{2})}].$ 
	\end{enumerate}	By Proposition \ref{p:trL}, we can pass the derivative inside the trace in $(2).$ Both $(1)$ and $(2)$ follow from (\ref{eq:exdi2}) in Proposition \ref{p:trL} as $\sum_{L=1}^{\infty} \frac{1}{L!}\Con^L L^n L^{\frac{L}{2}}\zeta^{L\delta-n}\exp(-tL\hq(\delta))<\infty$ for each $\zeta\in [1,e^{t\bq(\frac{s}2)}]$.
\end{proof}

Now, with the  results from Lemmas \ref{p:trL} and \ref{p: d-s}, we are poised to justify the interchanges of operations leading to \eqref{eq:bfinal}. 
\begin{proposition} \label{p: s-i}
For fixed $n, L \in \Z_{\geq 0}$, $\zeta \in [1, e^{t\bq(\frac{s}{2})}]$ and $t > 0$, 
\begin{equation} \label{eq: intrchge}
\int_1^{e^{t\bq(\frac{s}{2})}}\zeta^{-\alpha}	\partial_{\zeta}^n \bigg[1+\sum_{L = 2}^{\infty}\tr(K_{\zeta,t}^{\wedge L}) \bigg] \d\zeta=\sum_{L =2}^{\infty}\sum_{\vec{m}\in \mathfrak{M}(L, n)}\binom{n}{\vec{m}}\frac{1}{L!}\int_1^{e^{t\bq(\frac{s}{2})}}\zeta^{-\alpha} I_{\zeta}(\vec{m})\d\zeta.
\end{equation}
\end{proposition}
\begin{proof} Thanks to Lemma \ref{p: d-s} we can switch the order of derivative and sum to get
\begin{align*}
	\mbox{l.h.s.~of \eqref{eq: intrchge}} = \int_1^{e^{t\bq(\frac{s}2)}} \sum_{L=2}^{\infty} \zeta^{-\alpha}  \partial_{\zeta}^n(\tr(K_{\zeta,t}^{\wedge L}))\d \zeta.
\end{align*}	
We next justify the interchange of the integral and the sum in above expression. Note that via the estimate in \eqref{eq:exdi2} we have
\begin{align*}
	\int_1^{e^{t\bq(\frac{s}2)}} \sum_{L=2}^{\infty} \zeta^{-\alpha}  |\partial_{\zeta}^n(\tr(K_{\zeta,t}^{\wedge L}))|\d \zeta \le \sum_{L=2}^{\infty} \frac1{L!}\Con^LL^nL^{\frac{L}2}\exp(-t\hq(\delta))\int_1^{e^{t\bq(\frac{s}2)}}\zeta^{L\delta-n-\alpha}\d\zeta <\infty.
\end{align*}
Hence Fubini's theorem justifies the exchange of summation and integration. Finally we arrive at r.h.s.~ of \eqref{eq: intrchge} by using the higher order derivative identity (see \eqref{eq: exdi}) from Proposition \ref{p:trL}.

\end{proof}

\subsection{Proof of Proposition \ref{p:ho}}\label{pf: p:ho}
Finally, in this subsection we present the proof of Proposition \ref{p:ho} via obtaining an upperbound for $|\mathcal{B}_s(t)|$, defined in (\ref{eq:Calb}). 

\medskip

Recall $I_{\zeta}(\vec{m})$ from \eqref{def: Im}. We first introduce the following technical lemma that upper bounds the absolute value of the integral $\int_1^{e^{t\bq(\frac{s}{2})}}\zeta^{-\alpha}I_{\zeta}(\vec{m})\d \zeta$ and will be an important ingredient in the proof of Proposition \ref{p:ho}.
\begin{lemma}\label{l:itmdb}
Fix $s>0$ so that $\alpha:=s-\lfloor s \rfloor >0$. Set $n=\lfloor s\rfloor+1$. Fix $L\in \Z_{>0}$ with $L\ge 2$ and $\vec{m}\in \mathfrak{M}(L, n)$, where $\mathfrak{M}(L,n)$ is defined in \eqref{mln}.  There exists a constant $\Con = \Con(q,s)>0$ such that 
\begin{equation} \label{eq:imcase}
\int_1^{e^{t\bq(\frac{s}{2})}}\zeta^{-\alpha}|I_{\zeta}(\vec{m})|\d \zeta \le	\Con^L L^{\frac{L}{2}} \exp(-t\hq(s) -\tfrac{1}{\Con}t).
\end{equation} 
where $I_{\zeta}(\vec{m})$ is defined in \eqref{def: Im} and the functions $\bq$ and $\hq$ are defined in Proposition \ref{p:htau}.
\end{lemma}
\begin{proof}
	We split the proof into two steps as follows. Fix $L_0 = 2(n+1)$. In Step 1, we prove the inequality for when $2 \le L \le L_0$ and in Step 2, we consider the case when $L > L_0$. In both steps, we deform the $w$-contours in $I_{\zeta}(\vec{m})$ appropriately to achieve its upper bound.
	
	\medskip
	
	\noindent\textbf{Step 1. $2 \le L \le L_0.$}  Fix $\vec{m} = (m_1, \ldots, m_L) \in \mathfrak{M}(L,n),$  where $\mathfrak{M}(L,n)$ is defined in \eqref{mln} and set
		\begin{equation}\label{eq: delta}
		\rho_i := 
		\begin{cases}
			m_i + \frac{\alpha}{L} - \frac{1}{|\vec{m}|_{>0}} & \text{ if } m_i > 0\\
			\frac{\alpha}{L} & \text{ if } m_i = 0.
		\end{cases}
	\end{equation} 
	where recall that $|\vec{m}|_{>0}=|\{i\mid i\in \Z, m_i>0\}$.

	Recall the definition of $I_{\zeta}(\vec{m})$ in \eqref{def: Im}. Note that each $K_{\zeta,t}^{(m_i)}(w_i,w_j)$ (see \eqref{def: kerdrv}) are themselves are complex integral over $\delta+\i\R$. As $\alpha>0$ and $L\le L_0=2(n+1)$ we may take the $\delta$ appearing in the kernel in $K_{\zeta,t}^{(m_i)}$ is less than all the $\rho_i$'s. Note that this is only possible when $\alpha>0$. This is why we assumed this in the hypothesis here and as well as in the statement of Proposition \ref{p:ho}. 
	  
	  In what follows we show that the contours of $K_{\zeta,t}^{(m_i)}(w_i,w_j)$ followed by $w_i$-contours can be deformed appropriately without crossing any pole in $I_{\zeta}(\vec{m})$. 
Indeed for each $K_{\zeta,t}^{(m_i)}$ in $I_{\zeta}(\vec{m})$ we can write $$K_{\zeta,t}^{(m_i)}(w_i, w_j) =\frac{1}{2\pi \i} \int_{\rho_i - \i \infty}^{\rho_i + \i \infty}\Gamma(-u_i)\Gamma(1+u_i)(u_i)_n \zeta^{u_i-n}e^{f(u_i,w_i)}\frac{\d u_i}{w_j-\tau^{u_i} w_i}.$$ As each $\rho_i \in (0, m_i\vee1)$ (see \eqref{eq: delta}), by Remark \ref{r:pole}, the above equality is true as we do not cross any poles in the integrand. Ensuing this change, we claim that we can deform the $w_i$-contour to $\C(\tau^{1-\frac{\rho_i}{2}})$ one by one without crossing any pole in $I_{\zeta}(\vec{m})$. Similar to the argument given in the beginning of the proof of Proposition \ref{p:leading}, we note that as we deform the $w_i$-contours potential sources of poles in $I_{\zeta}(\vec{m})$ lie in the exponent $f(u_i,w_i):=\frac{(q-p)}{1+w_i\tau^{-1}}-\frac{(q-p)}{1+\tau^{u_i-1}w_i}$ (recalled from \eqref{eq:contour_fn}) and in the denominator $w_j- \tau^{u_i} w_i.$ 

Take $w_i \in \C(\tau^{1-\frac{\delta_i}{2}}),$ $\delta_i, \in [\delta, \rho_i]$, and $u_i\in  \rho_i+\i \R$. Observe that
$$|w_j-\tau^{u_i}w_i| \ge |w_j|-|\tau^{u_i}w_i| \ge \tau^{1-\frac{\delta_j}2}-\tau^{1+\rho_i-\frac{\delta_i}2}>0, $$
$$|1+w_i\tau^{-1}| \ge |w_i\tau^{-1}|-1 \ge \tau^{-\frac{\delta_i}2}-1, \quad |1+\tau^{u_i-1}w_i| \ge 1-|\tau^{u_i-1}w_i| \ge 1-\tau^{\rho_i-\frac{\delta_i}{2}}.$$
This ensures that each $w_i$-contour can be taken as $\C(\tau^{1-\frac{\rho_i}2})$ without crossing any pole.

	Permitting these contour deformations, we wish to apply Lemma \ref{l:kdbd}, \eqref{eq: idbd} specifically. Indeed we apply \eqref{eq: idbd} with $\rho,\delta\mapsto \rho_i$, $w\mapsto w'$, $w' \mapsto w_j$. Note that we indeed have $|w_j| \neq \tau^{1+\frac{\rho_i}2}$ here. We thus obtain
	\begin{equation}\label{eq:k}
	|K^{(m_i)}_{\zeta, t}(w_i,w_j)|\le \Con \zeta^{\rho_i - m_i}\exp(-t\hq(\rho_i)).
	\end{equation}
	Here,  $\Con $ is supposed to be dependent on $m_i,\rho_i,$ and $q$. Note that $\rho_i$ are in turn dependent on $m_i$, $s$ and $L$. Since $L$ is at most $L_0=2(n+1)$, there are at most finitely many choices of $m_i$'s which in turn produced finitely many choices of $\rho_i$'s. As $s$ is fixed, all of the $\rho_i$'s are uniformly bounded away from 0. Hence we can choose the constant $\Con$ to be dependent only $s$ and $q$ (recall that $n$ is also dependent on $s$). 
	
Observe that as $\vec{m}\in \mathfrak{M}(L,n)$ defined in \eqref{mln}, we have $\sum m_i=n$ and consequently $\sum \rho_i=n-1+\alpha=s$. In view of the estimate in \eqref{eq:k} and the definition of $I_{\zeta}(\vec{m})$ from \eqref{def: Im}, by Hadamard's inequality, we obtain  

	\begin{equation*}
	|I_{\zeta}(\vec{m})| \leq \Con^L L^{\frac{L}{2}}\zeta^{s-n}\exp\left(-t\sum_{i =1}^L\hq(\rho_i)\right) = \Con^L L^{\frac{L}{2}}\zeta^{-1 + \alpha}\exp\left(-t\sum_{i =1}^L\hq(\rho_i)\right). 
	\end{equation*}
Thus
	\begin{equation}\label{eq: hoitmd}
	\int_1^{e^{t\bq(\frac{s}{2})}}\zeta^{-\alpha}|I_{\zeta}(\vec{m})|\d\zeta \leq\Con^L L^{\frac{L}{2}}\exp\left(-t\sum_{i =1}^L\hq(\rho_i)\right)\int_1^{e^{t\bq(\frac{s}{2})}} \zeta^{-1}\d \zeta.
	\end{equation}
	Observe that $\int_x^y\zeta^{-1} d\zeta = \log\frac{y}{x}$. We appeal to the subadditivity $\hq(x)  + \hq(y) > \hq(x+y)$ in Proposition \ref{p:htau} to get that $\sum_{i=1}^L \hq(\rho_i) \ge \hq(s-\rho_1)+\hq(\rho_1)$. Note that here we used the fact that $L\ge 2$. This leads to
	\begin{align}\label{eq:to}
		\mbox{r.h.s.~of \eqref{eq: hoitmd}} \le \Con^LL^{\frac{L}2}t\bq(\tfrac{s}2)\exp(-t\hq(s))\exp(-t(\hq(s-\rho_1)+\hq(\rho_1)-\hq(s)))
	\end{align}
	Note that from \eqref{eq: delta}, $\rho_i \ge \frac{\alpha}{L} \ge \frac{\alpha}{L_0}$, this forces $\frac{\alpha}{L_0} \le s-\rho_1,\rho_1 \le s-\frac{\alpha}{L_0}$. Appealing to the strict subadditivity in \eqref{eq:diff} gives us 
that $\hq(s-\rho_1)+\hq(\rho_1)-\hq(s)$ can be lower bounded by a constant $\frac{1}{\Con}>0$ depending only on $s$ and $q$. Adjusting the constant $\Con$ we can absorb $t\bq(\frac{s}{2})$ appearing in r.h.s.~of \eqref{eq:to}, to get \eqref{eq:imcase}, completing our work for this step.

\medskip
	
	\noindent\textbf{ Step 2. $L > L_0$.} Fix $\vec{m}= (m_1, \ldots, m_L)\in \mathfrak{M}(L, n).$  Recall the definition of $I_{\zeta}(\vec{m})$ in \eqref{def: Im}. Note that each $K_{\zeta,t}^{(m_i)}(w_i,w_j)$ (see \eqref{def: kerdrv}) are themselves are complex integral over $\delta+\i\R$.  Here we set $ \delta= \min(\frac{1}{2}, \frac{s}{2})$.  Thanks to \eqref{eq:exdi2} we have
	\begin{equation*}
	|I_{\zeta}(\vec{m})| \leq \Con^L L^{\frac{L}{2}}\zeta^{L\delta -n}\exp(-tL\hq(\delta)),
	\end{equation*}
	where the constant $\Con$ depends only on $n,\delta,$ and $q$ and thus only on $s$ and $q$. This leads to
	\begin{equation}\label{eq: shoitmd}
	\int_1^{e^{t\bq(\frac{s}{2})}}\zeta^{-\alpha}|I_\zeta(\vec{m})|\d\zeta \le\Con^L L^{\frac{L}{2}}\exp(-tL\hq(\delta))\int_1^{e^{t\bq(\frac{s}{2})}} \zeta^{-\alpha-n + L\delta}\d \zeta.
	\end{equation}
	Recall that $s=n-1+\alpha$. As $L\ge 2(n+1)$ and $\delta=\min(\frac12,\frac{s}2)$ we have $L\delta-n-\alpha>0$ in this case. Thus, we can upper bound the integral in \eqref{eq: shoitmd} to get 
	\begin{equation}\label{eq: shoitmd2}
	\mbox{r.h.s.~of \eqref{eq: shoitmd}} \le \Con^L L^{\frac{L}{2}}\exp(-tL\hq(\delta))\frac{\exp(t\bq(\frac{s}{2})(-s + L\delta))}{-s + L\delta}.
	\end{equation}
We incorporate $\frac{1}{-s + L\delta}$ into the constant $\Con$, Recall the definition of $\bq(x)$ from Proposition \eqref{p:htau}. We have $x\bq(x)=\hq(x)$. As $\bq(x)$ is strictly decreasing for $x>0$, (Proposition \ref{p:htau} \ref{a}, \ref{b}) we have
	\begin{equation*}
	\begin{split}
	\mbox{r.h.s.~of \eqref{eq: shoitmd2}} &\le \Con^L L^{\frac{L}{2}}\exp(-2t\hq(\tfrac{s}{2}) -tL\delta(\bq(\delta)-\bq(\tfrac{s}{2})))\\& \leq \Con^L L^{\frac{L}{2}}\exp(- 2t\hq(\tfrac{s}{2})) \le \Con^L L^{\frac{L}{2}}\exp(- t\hq(s)-\tfrac{1}{\Con}t),
	\end{split}
	\end{equation*}
	where the last inequality above follows from \eqref{eq:diff} by observing that by subadditivity we can get a constant $\Con=\Con(q,s)>0$ such that $2\hq(\frac{s}{2})-\hq(s)\ge \frac1{\Con}$. This completes the proof. 
\end{proof}

With Lemma \ref{l:itmdb}, we are now ready to prove Proposition \ref{p:ho}.
\begin{proof}[Proof of  Proposition \ref{p:ho}]
	Recall the definition of  $\mathcal{B}_s(t)$ as defined in (\ref{eq:Calb}). Appealing to \eqref{eq: bst} and Proposition \eqref{p: s-i} we get that
	\begin{align}\label{bsl}
		|\mathcal{B}_{s}(t)| = \sum_{L =2}^{\infty}\frac{1}{L!}\sum_{\vec{m}\in \mathfrak{M}(L, n)}\binom{n}{\vec{m}}\int_1^{e^{t\bq(\frac{s}{2})}}\zeta^{-\alpha}|I_{\zeta}(\vec{m})|\d\zeta
	\end{align}
	 Note that $\binom{n}{\vec{m}}$ is bounded from above by $n!$, and by \eqref{eq:rec} we have $|\mathfrak{M}(L, n)| \leq L^n$. Applying these inequalities along with the estimate in Lemma \ref{l:itmdb} we have that 
	\begin{equation*}
	\begin{split}
	\mbox{r.h.s.~of \eqref{bsl}}  \le \exp(-t\hq(s)-\tfrac{1}{\Con}t)\sum_{L=2}^{\infty} \frac{1}{L!}\Con^L L^{\frac{L}{2}}L^n
	\end{split}
	\end{equation*}
	for some constant $\Con = \Con(q, s)>0.$
	By Stirling's formula, $\sum_{L=2}^{\infty}\frac{1}{L!}\Con^L L^{\frac{L}{2}}L^n$ converges and hence adjusting the constant $\Con$, we obtain \eqref{eq:ho} completing the proof of the proposition. 
\end{proof}

	\appendix

\section{Comparison to TASEP}\label{app} In this section, we compute explicit expression for the upper tail rate function for TASEP (ASEP with $q=1$) with step initial data and show that it matches with general ASEP rate function $\Phi_{+}$ defined in \eqref{eq:ldp}. 

Indeed, the large deviation problem for TASEP is already solved in \cite{joh} and is formulated in terms of Exponential Last Passage Percolation (LPP) model (Theorem 1.6 in \cite{joh}). 

In order to state the connection between TASEP and Exponential LPP, we briefly recall the Exponential LPP model. Let $\Pi_N$ be the set of all upright paths $\pi$ in $\Z_{>0}^2$ from $(1,1)$ to $(N,N)$. Let $w(i,j), (i,j)\in \Z_{>0}^2$ be independent exponential distributed random variables with parameter $1$. The last passage value for $(N,N)$ is defined to be
\begin{align*}
	\calH(N):=\max\big\{\sum_{(i,j)\in \pi} w(i,j); \pi\in \Pi_N\big\}.
\end{align*}
As with the ASEP, for TASEP, we also set $H_0^{q=1}(t)$ to be the number of particles to the right of origin at time $t$. It is well known (see \cite{joh} for example) that $H_0^{q=1}(t)$ is related to the last passage value $\calH(N)$ in the following way
\begin{align}\label{eq:lpp-rel}
	\P\left(-H_0^{q=1}(t)+\tfrac{t}4\ge \tfrac{t}{4}y\right)= \P(\calH(M_t)\ge t), \quad \mbox{where } M_t=\lfloor \tfrac{t}{4}(1-y) \rfloor +1.
\end{align}
\begin{theorem}\label{thm:tldp}
	For $y\in (0,1)$ we have
	\begin{align}\label{eq:tldp}
		\lim_{t\to\infty}\frac1t\log\P\left(-H_0^{q=1}(t)+\tfrac{t}4\ge \tfrac{t}{4}y\right)=-\Phi_{+}(y).
	\end{align}
	where $\Phi_{+}$ is defined in \eqref{eq:ldp}.
\end{theorem}

The idea of the proof of Theorem \ref{thm:tldp} is to use large deviation principle for $\calH(N)$ which appears in Theorem 1.6 in \cite{joh} followed by an application of the relation \eqref{eq:lpp-rel}. The only impediment is that the Johansson result appears in a variational form. 

Let us recall Theorem 1.6 in \cite{joh}. According to Eq (1.21) in \cite{joh} (with $\gamma=1$), the upper tail of $\calH(N)$ satisfy the following large deviation principle
\begin{align}\label{eq:lldp}
	\lim_{N\to\infty}\tfrac1N\log\P(\calH(N)\ge Nz)=-J(z), \quad z\ge 4.
\end{align}	
where the rate function $J$ is given by
\begin{align}\label{eq:gv}
	J(t):=\inf_{x\ge t} [G_V(x)-G_V(4)], \ t\ge 4, \ \ \ G_V(x):=-2\int_{\R} \log|x-r|\d\mu_V(r)+V(r), \  x\ge 4.
\end{align}
Here $V(x)=x$ is defined on $[0,\infty)$, and the measure $\mu_V$ is the unique minimizer of  $I_V(\mu)$ over $\mathcal{M}(\R_{\ge 0})$, the set of probability measures on $[0,\infty)$. $I_V(\cdot)$ is known as the \textit{logarithmic entropy in presence of the external field $V$} and is given by
\begin{align*}
	I_V(\mu):= -\iint_{\R^2} \log|x_1-x_2|\d\mu(x_1)\d\mu(x_2)+\int_{\R} V(x)\d\mu(x), \quad  \mu \in \mathcal{M}(\R_{\ge0}).
\end{align*}
The logarithmic entropy $I_V(\mu)$ is well studied in both mathematical and physics literature and has several applications to random matrix theory and related models. We refer to \cite{safftotik} and \cite{hp} and the references there in for more details.

The form of the rate function defined in \eqref{eq:gv} is not exactly same as in \cite{joh}. However, one can show the rate function $J$ defined in \eqref{eq:gv} is same as Eq (2.15) in \cite{joh} using the properties of minimizing measure (see Theorem 1.3 in \cite{safftotik} or Eq (1.6) in \cite{dragsaff}). Such an expression for the rate function is derived using Coulomb gas theory. We refer to \cite{joh}, \cite{feral}, and \cite{dd21} for treatment on the LDP problems of such nature.

\begin{proof}[Proof of Theorem \ref{thm:tldp}] For clarity we split the proof into two steps. 
	
	\medskip
	
	\noindent\textbf{Step 1.} 	
	We claim that $J$ defined in \eqref{eq:gv} has the following explicit expression.
	\begin{align}\label{eq:jfn}
		J(t)=\sqrt{t^2-4t}-2\log\frac{t-2+\sqrt{t^2-4t}}{2}, \quad t\ge 4.
	\end{align}
	We will prove \eqref{eq:jfn} in Step 2. Here we assume its validity and conclude the proof of \eqref{eq:tldp}. 
	
	Towards this end, fix $y\in (0,1)$ and $K$ large enough such that $[y-\frac1K,y+\frac1K] \subset (0,1)$. Recall the definition of $M_t$ from \eqref{eq:lpp-rel}. Note that for all large enough $t$, we have $\tfrac{4}{1-y+K^{-1}}M_t \le t \le \tfrac{4}{1-y-K^{-1}}M_t$. Thus
	\begin{align*}
		\P\big(\calH(M_t)\ge \tfrac{4}{1-y+K^{-1}}M_t\big) \ge \P\left(-H_0^{q=1}(t)+\tfrac{t}4\ge \tfrac{t}{4}y\right) \ge \P\big(\calH(M_t)\ge \tfrac{4}{1-y-K^{-1}}M_t\big).
	\end{align*}
	Taking logarithms on each side, dividing by $M_t$ and then taking $t\to\infty$ we get
	\begin{equation}\label{eq:tasep-lpp}
		\begin{aligned}
			-J\big(\tfrac{4}{1-y+K^{-1}}\big) & \ge \limsup_{t\to \infty}\frac1{M_t}\P\left(-H_0^{q=1}(t)+\tfrac{t}4\ge \tfrac{t}{4}y\right) \\ & \ge \liminf_{t\to \infty}\frac1{M_t}\P\left(-H_0^{q=1}(t)+\tfrac{t}4\ge \tfrac{t}{4}y\right) \ge -J\big(\tfrac{4}{1-y-K^{-1}}\big).
		\end{aligned}
	\end{equation}
	where we used the upper tail large deviation principle for $\calH(N)$ from \eqref{eq:lldp}.
	Observe that $\frac{M_t}{t} \to \frac{1-y}{4}$, and using \eqref{eq:jfn} we see that
	\begin{align*}
		\frac{1-y}{4}J\big(\frac{4}{1-y}\big)=\frac{1-y}{4}\left(\frac{4\sqrt{y}}{1-y}-2\log\frac{2(1+y)-4\sqrt{y}}{2(1-y)}\right)=\Phi_{+}(y),	
	\end{align*}
	where $\Phi_{+}$ is defined in \eqref{eq:ldp}. Thus taking $K\to \infty$ in \eqref{eq:tasep-lpp} we arrive at \eqref{eq:tldp}.
	
	\medskip
	
	\noindent\textbf{Step 2.} We now turn our attention to prove \eqref{eq:jfn}. 
	It is well known that for $V(x)=x$, the minimizer $\mu_V$ is given by the \textit{Marchenko-Pastur} measure (see Equation 3.3.2 and Proposition 5.3.7 in \cite{hp} with $\lambda=1$):
	\begin{align*}
		\d\mu_V(x) & = \frac{\sqrt{4x-x^2}}{2\pi x}\ind_{x\in[0,4]}\d x.
	\end{align*}
	Recall $G_V(x)$ defined in \eqref{eq:gv}. Using the Cauchy Transform for $\mu_V$ (see the last unnumbered equation in Page 200 of \cite{hp}) we get that for $x > 4$, 
	\begin{align*}
		\frac{\d}{\d x}\int \log|x-r|\d\mu_V(r) = \frac12-\frac{\sqrt{x^2-4x}}{2x},
	\end{align*}
	which implies $G_V(z)-G_V(4)  = \int_4^z \frac{\sqrt{x^2-4x}}{x}\d x.$
	Thus $G_V(z)-G_V(4)$ is strictly increasing in $y$ and whence by \eqref{eq:gv} we have
	\begin{align*}
		J(t)=\int_4^t \frac{\sqrt{x^2-4x}}{x}\d x.
	\end{align*}
	To compute the above integral, we make the change of variable $x \mapsto \frac{(z+1)^2}z$ so that $\d x= (1-\frac1{z^2})\d z$ and $x^2-4x= \frac{(z^2-1)^2}{z^2}$. Set $a=\frac{t-2}{2}+\frac{\sqrt{t^2-4t}}{2}$ to get
	\begin{align*}
		\int_4^t \frac{\sqrt{x^2-4x}}{x}\d x =  \int_1^{a} \frac{(z-1)^2}{z^2}\d z & = \left[z-\frac1z-2\log z\right]_1^{a} = a-\frac1a-2\log a.
	\end{align*}
	Plugging the value of $a$ we get \eqref{eq:jfn} completing the proof.
\end{proof}

	\bibliographystyle{alphaabbr}		
	\bibliography{asep}
\end{document}